\documentclass[11pt]{article}
\usepackage[a4paper,margin=3cm]{geometry}
\usepackage[utf8]{inputenc} 
\usepackage[T1]{fontenc} 
\usepackage{lmodern}
\usepackage{todonotes}
\usepackage[autolanguage]{numprint} 
\usepackage{hyperref} 
\usepackage{graphicx} 
\usepackage[all]{xy}
\usepackage{amsmath,amssymb,amsthm}
\newtheorem{assumption}{Assumption}

\newcommand\NN{\mathbb{N}} 
\newcommand\RR{\mathbb{R}} 
\newcommand\CC{\mathbb{C}} 
\newcommand\PP{\mathbb{P}} 
\newcommand\EE{\mathbb{E}} 

\newcommand\DD{\mathbb{D}}

\newtheorem{theorem}{Theorem}[section]%
\newtheorem{corollary}[theorem]{Corollary}%
\newtheorem{proposition}[theorem]{Proposition}%
\newtheorem{lemma}[theorem]{Lemma}%

\newtheorem{definition}[theorem]{Definition}%
\newtheorem{remark}[theorem]{Remark}%

\author{Raphael Butez\footnote{This project has received funding from the European Research Council (ERC) under the European Union’s Horizon 2020 research and innovation programme (grant agreement No. 692452).}\,\, and David García-Zelada}

\title{Extremal particles of two-dimensional Coulomb gases and random polynomials on a positive background}

\begin{document}
\maketitle

\begin{abstract}
We study the outliers for two models which have an interesting connection. On the one hand, we study a specific class of planar Coulomb gases which are determinantal. It corresponds to the case where the confining potential is the logarithmic potential of a
radial probability measure. 
On the other hand, we study the zeros of random  polynomials that 
appear to be closely related to the first model. 
Their behavior
far from
the origin is shown to depend
only on the decaying properties
of the probability measure generating
the potential. A similar
feature is observed for their behavior near the
origin. Furthermore, in some cases,
the appearance of outliers is observed,
and the zeros of random polynomials
and the Coulomb gases are seen to exhibit 
exactly the same behavior, which is related to
the unweighted Bergman kernel. 

\vspace{2.5mm}

\noindent
\textit{2020 MSC:}
\small
60G55; 82B21;
60F05; 60K35; 30C15.

\vspace{2.5mm}

\end{abstract}

\section{Introduction}
\subsection{Coulomb gases and random polynomials}

We will be interested in
a system of 
$n$ interacting particles at
equilibrium whose positions,
$x_1,\dots,x_n \in \nolinebreak \mathbb{C}$,
follow the law

\begin{equation}\label{gaz}
\frac{1}{Z_n} \exp 
\left(-2  \left[ -\sum_{i<j}  \log |x_i-x_j| + (n+1) \sum_{i=1}^n V(x_i)  \right]
\right) {\rm d}\ell_{\mathbb{C}^n}(x_1, \dots, x_n),
\end{equation}
where $Z_n$ is a normalizing constant,
$V$ is a 
continuous real valued function,
called the confining potential,
and $\ell_{\mathbb{C}^n}$ is the Lebesgue measure on $\mathbb{C}^n$.
If we define the Hamiltonian 
\begin{equation}
\label{eq:Hamitonian}
(x_1,\dots,x_n) \mapsto
-\sum_{i<j} \log |x_i-x_j| + (n+1) \sum_{i=1}^n V(x_i), 
\end{equation}
then  \eqref{gaz} is the canonical Gibbs measure at inverse temperature $2$ associated to this Hamiltonian. 
The random element
$(x_1,\dots,x_n)$ 
is sometimes known as
a \textit{Coulomb gas}
or a \textit{two-component plasma}
since \eqref{eq:Hamitonian} can be interpreted 
as the electrostatic energy
of a system of confined electrically
charged particles.
This model is well-defined
for every $n$ as soon as
the potential $V$ satisfies
\begin{equation}
\label{potential}
\varliminf_{|z|\to \infty}
\left\{
V(z)-\log|z| \right\}> - \infty.
\end{equation}
Notice that if we had written
$n$
instead of $n+1$
 in front of $V$
in \eqref{gaz},  then
\eqref{potential}
would not be enough to assure
that \eqref{gaz} is well-defined.
In this article, we will only consider potentials of the form
\begin{equation}
\label{eq:potential}
 V^{\nu}(z) = \int_{\mathbb C} \log|z-w|
 {\rm d}\nu(w) 
\end{equation}
for a rotationally invariant probability measure $\nu$ such that its potential $V^{\nu}$ makes
sense and is finite. 
For this kind of potentials,
$V(z) - \log|z|$ has a finite
limit as $|z| \to \infty$ so
that \eqref{potential}
is satisfied but we cannot replace
$n+1$ by $n$ in front of
the potential in \eqref{gaz}.
In this article, this model will
be called a
\textit{jellium}, since it corresponds to the situation where $n$ classical electrons with unit negative charge
are attracted by a positively charged
distribution.
 This denomination is not standard and was discussed in \cite{ChafaiGarciaZeladaJung}.  
In our case, the positive
distribution is $(n+1)\nu$, which has
total charge $n+1$. 
If $(x_1,\dots,x_n)$ is a
jellium associated to $\nu$,
 Frostman's criterion \cite{SaffTotik}
 and standard large deviation
 principles (see \cite{GarciaZelada}, for instance) 
 imply the convergence
 of the sequence of empirical measures
 towards $\nu$, i.e.
\[ \mu_n:=\frac{1}{n} \sum_{k=1}^n \delta_{x_k} \xrightarrow[n \to \infty]{\mathrm{a.s.}} \nu.  \]
\begin{assumption}\label{assumption1}
The probability measure $\nu$ is rotationally invariant and satisfies
 \[ \int_{\mathbb C} 
 |\log |z|\, |\, {\rm d}\nu(z)< \infty.\]
This implies, in particular, that $V^{\nu}$ is finite everywhere and that
\[ \mathcal{E}(\nu):= - 
\int_{\mathbb C} V^{\nu}(z) {\rm d}\nu(z) =  
-\int_{\mathbb C^2}
\log|z-w| {\rm d}\nu(z) {\rm d}\nu(w) \in \RR .  \]
\end{assumption}
The key to our approach is the fact that 
Coulomb gases in the plane at inverse temperature $2$ are determinantal point processes.
We state a definition
and some facts
used in this article
in
Subsection \ref{sub:DPP}
in the appendix.
For a nice
introduction to this subject,
we suggest \cite{HoughKrisPeresVirag}.

Given a radial measure $\nu$, 
we can 
also consider random polynomials
\begin{equation} \label{RandomPoly}
P_n(z) = \sum_{k=0}^{n} a_k R_{k,n}(z)
\end{equation}
where the $a_k$'s are 
\text{i.i.d.} random variables and
$(R_{k,n})_{k \in \{0,\dots,n \}}$
 is a normalized basis of
 monomials in $\mathbb{C}_n[X]$ for the inner product
\begin{equation}\label{Scalar Product}
\langle P, Q \rangle_{n,\nu} 
= \int_\mathbb{C} P(z) \overline{Q(z)} e^{-2n V^{\nu}(z)} 
{\rm d}\nu(z),
\end{equation}
i.e.
$R_{k,n}(z) = z^k/
\sqrt { \langle X^k,
X^k \rangle_{n , \nu}}$
where $X^k(z) = z^k$. 
In this article, we will always
assume that
$\mathbb P(P_n \neq 0)=1$ or,
equivalently, that
$\mathbb P(a_0 \neq 0) = 1$
so that the polynomial
$P_n$ has degree exactly $n$
and we will denote
its zeros, counted with multiplicity (and in
any order), 
by $z_1,\dots,z_n$.

\begin{assumption}\label{assumption2}
$a_0 \neq 0$ almost surely,
$a_0$ is not deterministic and 
$\EE(\log (1+|a_0|))<+\infty$.
\end{assumption}
This moment condition is classical for random polynomials. It ensures that the empirical measures of the zeros $\frac{1}{n}\sum_{k=1}^n \delta_{z_k}$ converges towards a deterministic measure \cite{IbragimovZaporozhets},\cite{BloomDauvergne} for many models of random polynomials.

\begin{remark}
Since $\nu$ is invariant
under rotations,
the
basis $(R_k)_{k \in \{0,\dots,n \}}$
is an orthonormal
basis
of $\mathbb{C}_n[X]$  for
the inner product
$\langle \cdot, 
\cdot \rangle_{n,\nu}$.
If the $a_k$'s are 
standard complex Gaussian random variables,
 i.e. if
$a_k \sim e^{-|z|^2} 
\mathrm d \ell_{\mathbb C}(z)/
\pi$, then the random polynomial
$P_n$ is just a
Gaussian random element
of 
$\mathbb{C}_n[X]$  with complex
variance
$\langle \cdot, 
\cdot \rangle_{n,\nu}$.
In this case,
the definition can be 
naturally extended
to non-radial $\nu$.
Nevertheless,
if the coefficients are 
not Gaussian, an
orthonormal basis should
be chosen
to use
the definition in \eqref{RandomPoly}.
\end{remark}

The Gaussian case of this 
model was introduced by Shiffman and Zelditch \cite{ShiffmanZelditch} in the context of random sections of line bundles. It covers the classical random polynomials ensembles, namely Kac polynomials\footnote{In
the definition of Kac polynomials, and throughout
the article, $S^1$ will
denote the unit circle.}, elliptic polynomials and (nearly) Weyl polynomials for specific choices of $\nu$.
\vspace{0mm}
\begingroup
\renewcommand{\arraystretch}{2}

\begin{center}
	\begin{tabular}{c||c|c|c|c|}
		
		Model &  Basis & Measure & Potential \\
		\hline \hline
		Kac &  $X^k$ & $\nu_{S^1}$ uniform on $S^1$ & $V^{\nu_{S^1}}(z)=\max( \log|z|,0)$\\
		\hline
		Elliptic &  $\sqrt{\binom{n}{k}}X^k$ & 
		$d\nu_{FS}(z)= \frac{d\ell_{\mathbb{C}}(z)}
{\pi(1+|z|^2)^2}$ & $V^{\omega_{FS}}(z)= \frac{1}{2}\log(1+|z|^2)$ \\
		\hline
		Nearly Weyl & $\frac{\sqrt{n^k}X^k}{\sqrt{k! - \int_n^{\infty} r^{k}e^{-r}{\rm d}r}}$ & 
		$d\nu_{\mathbb{D}}(z)= \frac{1_{|z|<1}d\ell_{\mathbb{C}}(z)}{\pi}$ & $V^{\nu_{\mathbb{D}}}(z)= 
	\begin{cases} \frac{1}{2}(|z|^2-1)  &\!\! \text{if } |z|<1 \\
	\log|z|  &\! \!\text{if } |z|\geq 1
		\end{cases} $\\
		\hline 
	\end{tabular}
\end{center}
\endgroup
The random polynomials that we called ``Nearly Weyl'' polynomials are not exactly the classical rescaled Weyl polynomials, which are usually defined as 
\[ P_n^{\text{Rescaled Weyl}}(z) =\sum_{k=0}^n \frac{\sqrt{n^k}}{\sqrt{k!}} a_k z^k \]
and which would be related to the
Lebesgue measure on the plane. The actual rescaled
Weyl polynomials
will be treated in
Theorem \ref{th:Weyl}. Figure \ref{Fig1} is a realization of the zeros of the three classical models of random polynomials.

\begin{figure}[!h]
	\centering
	\includegraphics[height=4.5cm]{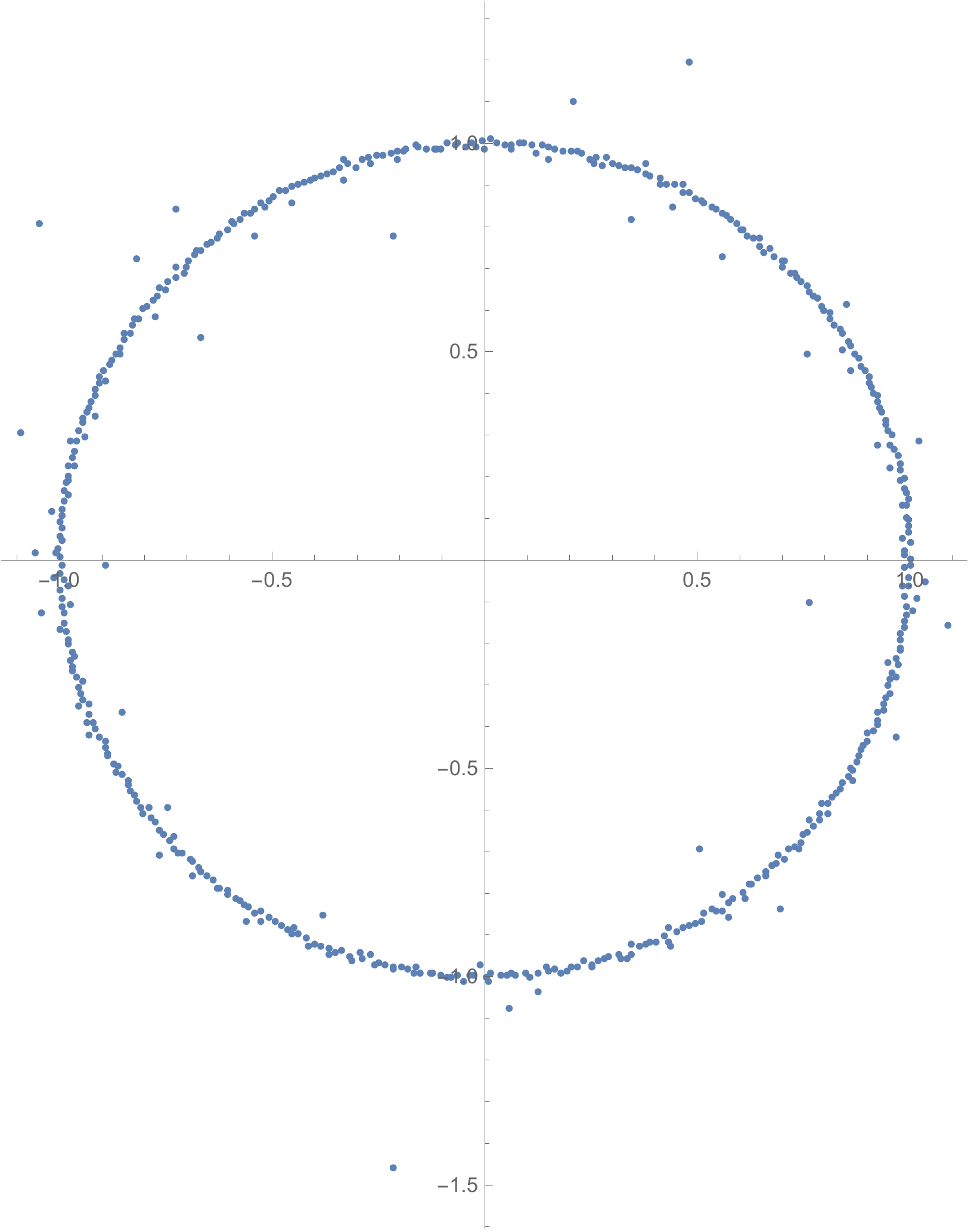}
	\includegraphics[height=4.5cm]{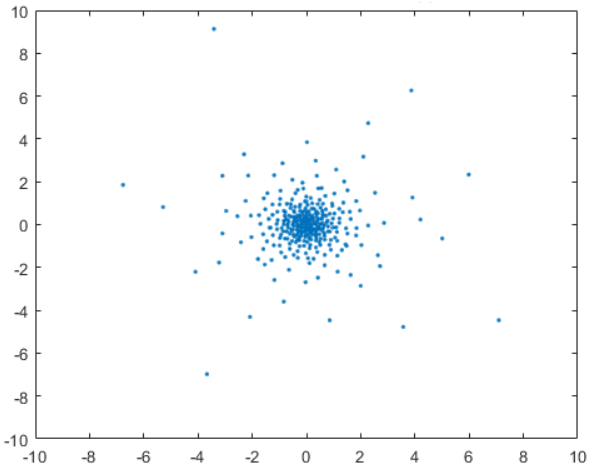}
	\includegraphics[height=4.5cm]{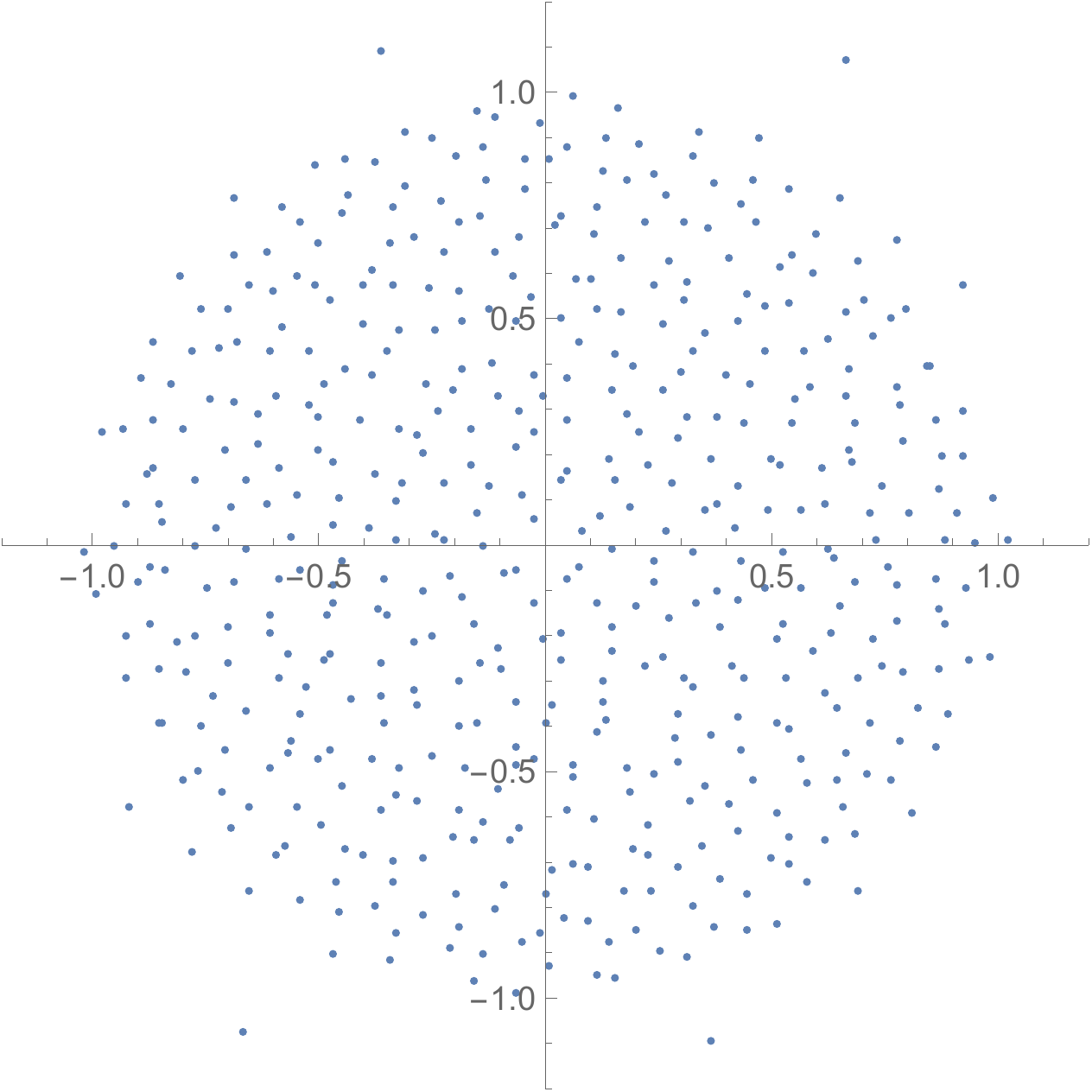}
	
	\caption{Roots of random polynomials A) Kac, B) Elliptic and C) Nearly Weyl, with 
	complex Gaussian coefficients.}\label{Fig1}
\end{figure}

The jellium and zeros of random polynomials present very similar behaviors: for these two models, the sequence of empirical measures $(\mu_n)_{n \in \NN^+}$ converges towards  $\nu$. Large deviation principles 
with speed $n^2$ but different rate functions are also valid. See \cite{ZeitouniZelditch,ButezZeitouni} for random polynomials and \cite{Hardy} for the jellium \footnote{Strictly 
speaking, \cite{Hardy}
fails to treat the inverse
temperature $2$ but the proof
still works for this case.
}. At the microscopic level, zeros of random polynomials seem to differ, as it is suggested by the article of Krishnapur and Virag \cite{KrishnapurVirag} which does not apply fully to this case. Figure \ref{Jellium} shows realizations of a jellium associated with the same measures $\nu$ as used in Figure \ref{Fig1}: A) $\nu_{S^1}$ the uniform measure on the unit circle, B) $\nu_{FS}$ the Fubini-Study measure, and C) $\nu_{\mathbb{D}}$ the uniform measure on the unit disk.
\begin{figure}[!h]
	\centering
	\includegraphics[height=4cm]{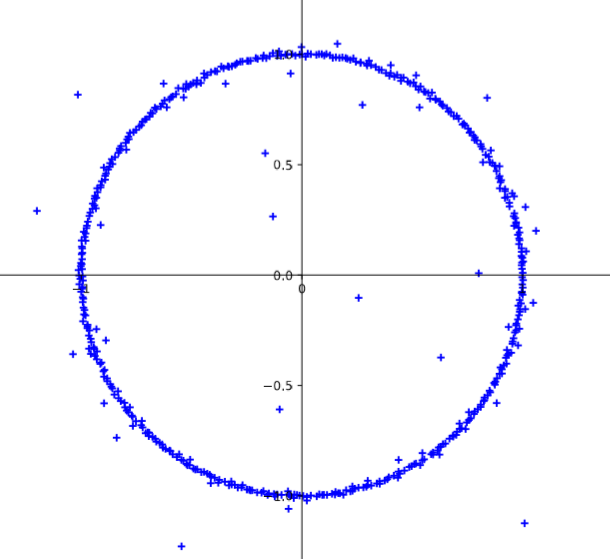}
	\includegraphics[height=4.5cm]{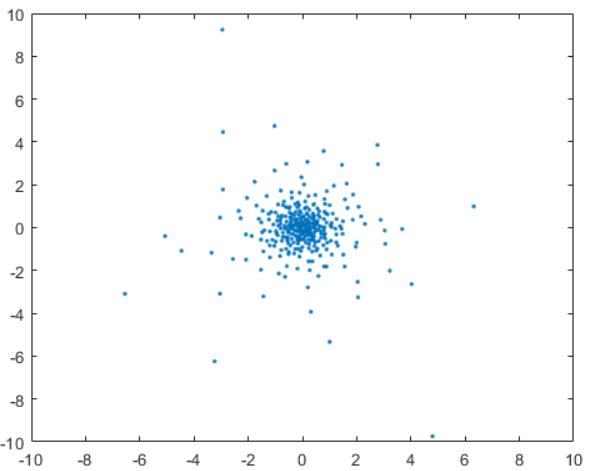}
	\includegraphics[height=4cm]{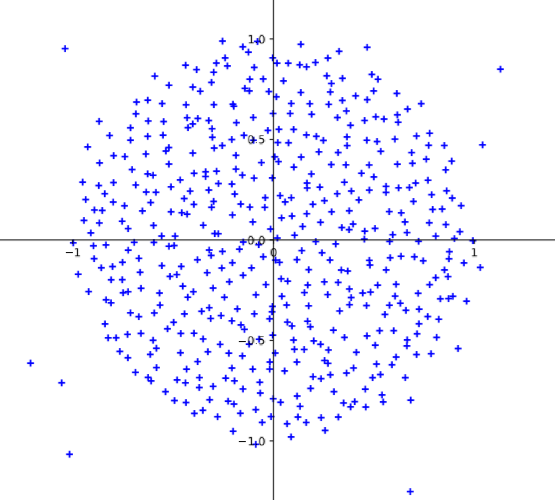}
	
	\caption{Jellium with background A) $\nu_{S^1}$, B) $\nu_{FS}$ et C) $\nu_\mathbb{D}$ .}\label{Jellium}
\end{figure}

\subsection{Results on the jellium}
In this section, we study the jellium. 
We are interested in the behavior of its
extremals, i.e.
either near the origin or near infinity.
We prove that, near the origin, 
the point process formed by the jellium
converges and its limit depends on the behavior
of $\nu$ near the origin. A similar feature
is observed near infinity. Outliers
are seen to
appear when $\nu$ gives zero charge
to a neighborhood of the origin
or to a neighborhood of infinity.  These outliers 
will converge to a Bergman point process
of a disk (neighborhood of zero) or of
the complement of the disk (neighborhood of infinity) that we introduce now.

\begin{definition}[Bergman point process]
\label{def:BergmanPointProcess}
	The Bergman point process of an open set $U \subset \CC$, written $\mathcal{B}_U$, is the determinantal point process on $U$ associated to the Bergman kernel of the set $U$. We will be interested
	in two particular cases: 
	The case of an open disk
	$D_R$
	of radius $R$
	where the Bergman kernel is
	\[ \forall z ,w \in D_R \quad  K_{D_R}(z,w)= 
	\frac{R^2}{\pi(R^2-z\bar{w})^2}\]
	and the case of the complement
	of the closed disk of
	radius $R$
	where the Bergman kernel \nolinebreak is
	\[ \forall z ,w \in
	\mathbb C \setminus \bar {D}_R \quad  K_{
		\CC \setminus \bar {D}_R}(z,w)= 
	\frac{R^2}{\pi(R^2
		z \bar w -1)^2}.\]
	
\end{definition}

For more information on the Bergman kernel, 
Bell \cite{Bell} gives a nice presentation on this topic.
An interesting connection to random
analytic functions is given in the work of 
Peres and Vir\'ag \cite{PeresVirag}.
More specifically, they prove that 
$\mathcal B_{\mathbb D}$
can be seen as the zeros of a Gaussian analytic
function. It can be shown that
$\mathcal B_{\mathbb D}$ is invariant 
in law under any conformal map 
of $\mathbb D$ and that, almost surely,
it has an infinite number of points and
every point of the unit circle is 
an accumulation point of $\mathcal B_{\mathbb D}$\footnote{This can be seen as a consequence of the
characterization
of the number of points
on a set
as a sum of independent Bernoulli
distributed random
variables
\cite[Theorem 4.5.3]{HoughKrisPeresVirag}
together with Borel-Cantelli lemma
for independent events.}.

\begin{theorem}[Outliers for the jellium]\label{JelliumThm}
Let $\nu$ be a probability measure 
	that satisfies \emph{Assumption \ref{assumption1}}
	and let $(x_1,\dots,x_n)$ be a jellium associated to $\nu$. Let $U$ be a connected component of $\CC \setminus \nolinebreak \mathrm{supp}\, \nu$ and suppose
that it is either an open disk or the complement of a 
closed disk.
Then the sequence of point processes $\{ x_k \text{ such that } x_k \in U \}$ converges weakly as $n$ goes to infinity towards the Bergman point process $\mathcal{B}_U$ of $U$.
\end{theorem}
The topology on (deterministic) point
process is reminded in the appendix, Section \ref{section:appendix},
for convenience of the reader. 
This theorem gives the universality of the outliers of the jellium since 
they only depend on the domain considered. In Figure \ref{Jellium}, the outliers outside of the unit disk for a jellium
associated to $\nu_{S^1}$ or $\nu_{\mathbb D}$
 have the same limiting behavior.

If $\nu$ has compact support, 
the convergence of the
point process outside a disk in 
Theorem \ref{JelliumThm} does not immediately imply the convergence of $\max_{k \in \{1,\dots,n\}} 
|x_k|$ towards a universal limiting random variable. Nevertheless,
it can be obtained
by an analysis of the minima in
the inverted model 
and it is stated in 
the next corollary.

\begin{corollary}[Particle of extremal modulus]\label{CorollaireJellium}
Let $\nu$ be a radial, compactly supported, probability measure satisfying \emph{Assumption \ref{assumption1}}. Let $R$ be the outer radius of the support 
of \nolinebreak $\nu$. Then, we have
\[\max_{k \in \{1, \dots, n\}} |x_k| 
\xrightarrow[n \to \infty]{\rm{law}}
R \, x_{\infty},\]
where $x_{\infty}$ is a random variable
taking values in $[1,\infty)$ with cumulative distribution function
			\[ \PP(x_{\infty}<t) 
			= \prod_{k=1}^{\infty} 
			\left(1- t^{-2k}\right).  \]
The law of
$R\, x_{\infty}$ is the same as the law of the 
maximum  modulus of the Bergman point process of $\CC \setminus 
\bar {D}_R$.
\end{corollary}
One can check\footnote{In fact, we can see that
$\mathbb P(\max_{k \in \{1, \dots, n\}}
|x_k| > t ) \sim C_n t^{-2}$ for some constant
$C_n>0$ and that
$\mathbb P(x_\infty > t ) \sim t^{-2}$.} that the random variable $x_{\infty}$ has
a finite expected value but infinite variance, as well
as the variable $\max_{k \in \{1, \dots, n\}}|x_k|$ for any $n$. This tells us that 
$\max_{k \in \{1, \dots, n\}}|x_k|$ is "often"
 far from $R$ as $n$ goes to infinity. This is a very different behavior from what was known in the context of strongly confining potentials \cite{ChafaiPeche}, for which Gumbel fluctuations were established for the maximum of the modulus at the edge of the support.

\begin{remark}[Universal behavior of the outer process and screening]
	The limiting law of the modulus of the extremal particle and the limiting point process do not depend on the choice of background $\nu$. Since most of the particles fill the disk 
according to the measure $\nu$, the outer particles "see" two canceling effects: on the one hand they are attracted by the positive background, but they are repelled by the negative charges which have nearly the same effect as the background. The universal behavior of the outer point process is a consequence of this competition.
Nevertheless, the behavior depends on the
coefficient $(n+1)$ at the left side of $V$ on 
\eqref{gaz} as can be seen in 
\cite{GarciaZelada3}.
\end{remark}

\begin{figure}
	\centering
	\includegraphics[width=.5\linewidth]{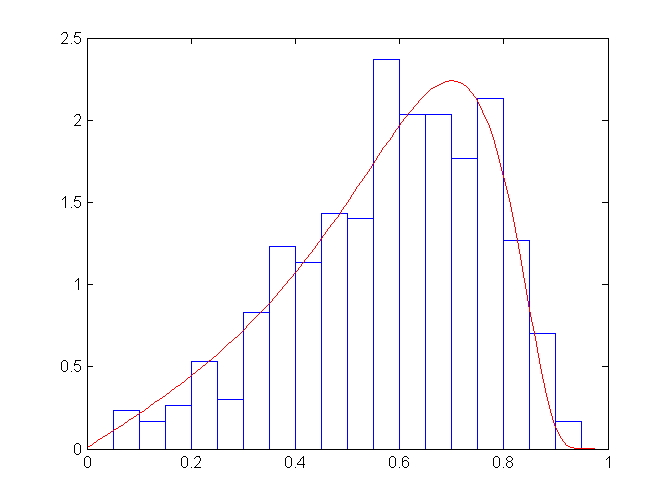}
	\caption{In blue: 
	Histogram for the lowest modulus in 
	a jellium associated to $\nu_{S^1}$
	 with $n=50$.
	In red: Density of the minimum modulus of
	$\mathcal{B}_{\mathbb{D}}$.}\label{MinJellium}
\end{figure}

Figure \ref{MinJellium}
displays the histogram of
a
sample of 
$(\max_{k \in \{1,\dots,n\}}|x_k| )^{-1}$ for the jellium associated to the uniform measure on the unit circle, and the density of the minima
of the Bergman point process of $\mathbb D$. Since
$\max_{k \in \{1,\dots,n\}}|x_k|$ and its limit are
 heavy tailed, it was much more convenient to represent the convergence of their inverses.

Theorem \ref{JelliumThm} described the behavior of the extremal particles if the region is uncharged. 
The next theorem tells us what happens when the extremal region is no longer uncharged.

\begin{theorem}[Extremal particles in the support]\label{BulkMinMax}
	Let $\nu$ be a measure 
	that satisfies \emph{Assumption \ref{assumption1}}
	and let $(x_1,\dots,x_n)$ be a jellium associated to $\nu$.
	
\textbf{At the origin:} Suppose there exists 
$\alpha >0$ and $\lambda>0$ such that
\[ \lim_{r \to 0} \frac{\nu(D_r)}{r^{\alpha}} = 
\lambda . \]
Then, $\{ n^{1/\alpha}x_1, \dots, n^{1/\alpha}x_n \}$
converges weakly towards the determinantal point process on $\mathbb{C}$ associated to the kernel
\begin{equation}
\label{eq:pointprocessnoncompact}
K(z,w)= \sum_{k=0}^{\infty} 
b_k z^{k} \bar{w}^k 
e^{-\gamma |z|^{\alpha}}e^{-\gamma |w|^{\alpha}} .
\end{equation}	 
	where
	\[\gamma=\frac{\lambda}{\alpha}
	\, \, \mbox{ and } \, \, b_k^{-1} = 
	2 \pi \int_0^\infty 
	r^{2k+1} e^{-2\gamma r^{\alpha}} {\rm d}r
	= 2\pi\,
	\Gamma\left(\frac{2k+2}{\alpha} \right) . \]	
	
	\textbf{At infinity: }
	Suppose
	there exists $\alpha >0$ and $\lambda>0$ such that
\begin{equation}
\label{eq:HypInfty}
 \lim_{r \to \infty} r^{\alpha}\nu(
	\hspace{0.5pt} \CC \setminus D_r) = \lambda.
\end{equation}
	Then
	$\{ n^{-1/\alpha} x_1, \dots, 
	n^{-1/\alpha} x_n \}$,
	considered
	as a point process
	on $\CC \setminus \{0\}$,
	converges weakly towards
	the determinantal point process
	on $\CC \setminus \{0\}$
	associated to the kernel \[K(z,w)=\sum_{k=0}^\infty 
	\frac{b_k}{(z \bar w)^{k+2}} 
	e^{-\gamma |z|^{-\alpha}}
	e^{-\gamma |w|^{-\alpha}}\]
	where
	\[\gamma=\frac{\lambda}{\alpha}
	\, \, \mbox{ and } \, \, b_k^{-1} = 
	2 \pi \int_0^\infty 
	r^{2k+1} e^{-2\gamma r^{\alpha}} {\rm d}r
	= 2\pi\,
	\Gamma\left(\frac{2k+2}{\alpha} \right) . \]
	
\textbf{Convergence of the maximum:}
	Under the same hypothesis
	\eqref{eq:HypInfty}, we have that
	\[ \frac{1}{n^{1/\alpha}} \max_{k \in \{1, \dots, n\}} |x_k| \xrightarrow[n \to \infty]{\mathrm{law}}
	\lambda^{1/\alpha} 
	x_{\infty} , \]
	where $x_{\infty}$ is a random variable with cumulative distribution function
\[\mathbb P(x_\infty \leq t) 
= \prod_{k=1}^\infty \frac{\Gamma
	\left(\frac{2k}{\alpha},
	\frac{2}{\alpha t^{\alpha} } 
	\right)}
{\Gamma
	\left(\frac{2k}{\alpha} \right)}\]
and the $\Gamma$ with two arguments
	denotes the upper incomplete gamma
	function.
	
\end{theorem}

Theorem \ref{BulkMinMax}
extends the corresponding result of Jiang and Qi \cite[Theorem 1]{JiangQi} on the spherical ensemble. Note that if $\nu$ has a positive density at the origin then the previous result applies with $\alpha =2$ and $\lambda$ equals to $\pi$
times the density at the origin. 
We recover the infinite Ginibre point process
in that case. 
The main interest of the result is not the explicit limiting random variables but its universality 
and the fact that no further
regularity is needed for \nolinebreak $\nu$.

\subsection{Results on random polynomials}
In this section, we present results on the extremal zeros of random polynomials associated to a background measure $\nu$ which are the counterparts of the results obtained for the jellium in the previous section. The results are very close to what was obtained before and are presented in the same order. 

\begin{theorem}[Outliers for random polynomials]\label{PolyThm}
Let $\nu$ be a probability measure satisfying \emph{Assumption} \ref{assumption1} and let
 $z_1,\dots,z_n$ be the zeros of a random polynomial $P_n$ 
 (given by 
 \eqref{RandomPoly}) such
 that $a_0$ satisfies \emph{Assumption} \ref{assumption2}. Let $U$ be a connected component of $\CC \setminus \nolinebreak \mathrm{supp}\, \nu$ and suppose
that it is either an open disk or the complement of a 
closed disk.

\textbf{Disk case:} If $U=D_R$
for some $R>0$, then
\[ \{ z_k 
\text{ such that } z_k \in D_R  \} 
\xrightarrow[n \to \infty]{\mathrm{law}} 
R \cdot \left\{ z \in \mathbb D
\text{ such that } \sum_{k=0}^{\infty} a_k z^k=0   \right\}.  \]

\textbf{Complement of a disk case:} 
If $U=\mathbb C 
\setminus \bar D_R$
for some $R>0$, then
\[ \{ z_k \text{ such that } z_k \in 
\CC \setminus \bar D_R  \} \xrightarrow[n \to \infty]{\mathrm{law}} R
\cdot \left\{ z \in \mathbb C \setminus \mathbb D
\text{ such that } \sum_{k=0}^{\infty} a_k \frac{1}{z^k}=0 
\right\}.  \]
\end{theorem}

\begin{remark}[On the limiting random series]
Arnold \cite{Arnold} showed 
that the random power series 
\[\sum_{k=0}^{\infty} a_k z^k\]
has a radius of convergence equal to $1$ almost surely as soon as $\EE(\log(1+|a_0|))<\infty$. 
In the specific case of complex
Gaussian coefficients, Peres and Vir\'ag \cite{PeresVirag} showed that its zeros follow the same law as the Bergman point process of the unit disk. Hence, in the case of complex Gaussian coefficients, this result is exactly the same as the one for the jellium.
\end{remark}

\begin{corollary}[Zero of extremal modulus]\label{CorollairePoly}
Let $\nu$ be a probability measure satisfying \emph{Assumption} \ref{assumption1} and let 
$z_1,\dots,z_n$ be the zeros of a random polynomial $P_n$
 (given by 
 \eqref{RandomPoly}) such
 that $a_0$ satisfies \emph{Assumption} \ref{assumption2}. Let $R$ be the outer radius of the support of $\nu$, then 
\[ \max_{k \in \{1, \dots, n\}} |z_k| \xrightarrow[n \to \infty]{\mathrm{law}} R \, \max 
\left\{|z| 
\text{ such that } 
z \in \CC \setminus \DD \mbox{ and }
\sum_{k=0}^{\infty}a_k \frac{1}{z^k}=0
\right\} . \]
In particular,
if $a_0$ is a complex Gaussian random variable,
we have
\[\max_{k \in \{1, \dots, n\}} |z_k| 
\xrightarrow[n \to \infty]{\rm{law}}
R \, z_{\infty},\]
where $z_{\infty}$ is a random variable
taking values in $[1,\infty)$ with cumulative distribution function
			\[ \PP(z_{\infty}<t) 
			= \prod_{k=1}^{\infty} 
			\left(1- t^{-2k}\right).
			\] 

\end{corollary}
The limiting random variable that appears in the corollary is only explicitly known for complex Gaussian coefficients, where it is the maximum 
modulus of the Bergman point process of $D_R$, as in Corollary \ref{CorollaireJellium}. For general coefficients, 
this random variable, as well as 
$\max_{k \in \{1,\dots, n\}}|z_k|$ are heavy tailed  \cite{Butez3}. 
More precisely, if \[\varliminf_{t \to 0^+} t^{-m} \mathbb P(|a_0|<t) >0,\] then the $m$-th moment of $\max_{k \in \{1, \dots, n\}} |z_k|$ as well as the
$m$-th moment of its limit are infinite.
Figure \ref{MinKac} illustrates the convergence of
$(\max_{k \in \{1, \dots, n\}} |z_k|)^{-1}$ 
for Gaussian Kac polynomials (i.e.
when $\nu = \nu_{S^1}$ and $a_0$
is a complex Gaussian random variable)
towards the minimum modulus of the Bergman point process.

\begin{figure}[!h]
	\centering
	\includegraphics[width=.5\linewidth]{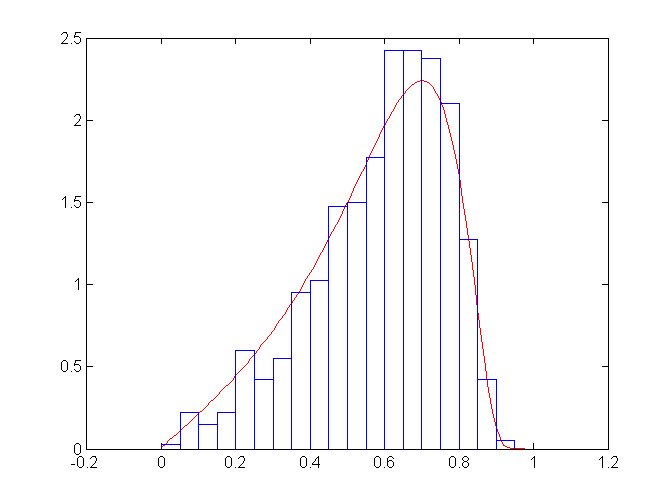}
	\caption{In blue: Histogram for the inverse of the largest root in modulus of Kac polynomials with complex Gaussian coefficients for $n=200$. 
	In red: Density of
	the minimum
	modulus of the Bergman point process
	of $\DD$.  } \label{MinKac}
\end{figure}

In the case where the support of $\nu$ is unbounded or contains the origin, we observe a phenomenon similar to the results of Theorem \ref{BulkMinMax} but for which the limiting random variable differs from 
the jellium case.
\begin{theorem}[Rescaled extremal roots of random polynomials]\label{BulkMinMaxPoly}
Let $\nu$ be a measure satisfying \emph{Assumption \ref{assumption1}} and
let
 $z_1,\dots,z_n$ be the zeros of a random polynomial $P_n$ 
 (given by 
 \eqref{RandomPoly}) such
 that $a_0$ satisfies \emph{Assumption} \ref{assumption2}.
 
{\bf At the origin:} If there exists $\alpha>0$ and $\lambda >0$ such that 
\begin{equation}
\label{eq:HypZero}
 \lim_{r \to 0} \frac{\nu(D_r)}{r^{\alpha}} = \lambda 
 \end{equation}
		then the point process $\{n^{1/\alpha}z_1,\dots, n^{1/\alpha}z_n \}$ converges
almost surely 
towards the roots of the 
following random entire function
			
			\[ f_{\alpha,\lambda}(z)= \sum_{k=0}^{\infty} \frac{a_k}{\Gamma(1+\frac{2k}{\alpha})^{1/2} } \left(\frac{\lambda}{\alpha}\right)^{k/\alpha} z^k,\] 
sometimes known as
the Mittag-Leffler random function.
		
{\bf At infinity: } If there exists $\alpha>0$ and $\lambda>0$ such that 
		\begin{equation}
		\label{eq:HypInfty2}
		 \lim_{r \to \infty } r^{\alpha} 
		\nu\, (\hspace{0.5pt} \CC \setminus D_r) = \lambda
		\end{equation}
		then
		the point process $\{n^{-1/\alpha}z_1,
		\dots,  n^{-1/\alpha}z_n\}$,
		seen as a point process
		on $\mathbb C \setminus \{0 \}$, 
		converges in law
		towards the inverse of the point
		process of the zeros of
		$f_{\alpha,\lambda}$.
		
\textbf{Convergence of the maximum:} Under the same hypothesis \eqref{eq:HypInfty2}, we also have
		\[\frac{1}{n^{1/\alpha}} 
		\max_{k \in \{1, \dots, n\}} |z_k| \xrightarrow[n \to \infty]{\mathrm{law}} 
		\max\{|z| \text{ such that } 
		z \in \CC\setminus \{0\} \mbox{ and }
		f_{\alpha,\lambda}(1/z)=0\}.\]
\end{theorem}
When $\alpha=2$ and $\lambda = 2$, which occurs for elliptic polynomials, we have
\[ f_{2,2}(z)=\sum_{k=0}^{\infty} a_k \frac{1}{\sqrt{k!}}z^k.  \] 
and,
if $a_0$ is a complex standard Gaussian 
random variable,
this random function is known as the 
planar Gaussian Analytic function. 
For different $\alpha$ and $\lambda$ but still
complex Gaussian coefficients, the random Mittag-Leffler functions are studied for the rigidity of their zero set \cite{KiroNishry}.

\begin{remark}
 In the special case $\nu = \nu_{S^1}$,
 the uniform measure on the unit circle, Theorem \ref{PolyThm} and Corollary \ref{CorollairePoly} are just a direct consequence of the result of Arnold \cite{Arnold}, while the identification of the limiting point process for complex Gaussian coefficients is exactly the result of Peres and Vir\'ag \cite{PeresVirag}. For all the other models of random polynomials, our result is new. There is no hope to observe the Bergman point process for non-Gaussian coefficients, as it was already noticed in \cite[Section 5]{TaoVu}. See \cite{Butez3} for
	a discussion of this non-universality
	in the case of the Kac polynomials.
\end{remark}

It is possible to extend Theorem \ref{PolyThm}
to the case where $\nu(\mathbb C) > 1$. Indeed,
the same methods that
will be used to prove 
Theorem \ref{PolyThm} would allow us to prove
the straightforward generalization.
Since it escapes
the main models of interest in this article,
we will only state the case
of Weyl polynomials which have an easier
and shorter proof.

\begin{theorem}[Extremal particles for the Weyl polynomials]
\label{th:Weyl}

Let $(a_k)_{k \in \mathbb N}$ be
a sequence of  i.i.d. random variables satisfying \emph{Assumption} \ref{assumption2}.
Let	 
  \[P_n(z) = 
\sum_{k=0}^n 
\frac{\sqrt{n^k}}{\sqrt {k!}} a_k z^k.\]
If  $z_1,\dots,z_n$ 
are the zeros of $P_n$ then
\[ 
\{ z_k \text{ such that } |z_k|>1\}  \xrightarrow[n\to \infty]{\mathrm{ law }} 
\left\{ z \in \CC
\setminus \bar{\mathbb D }
\text{ such that } \sum_{k=0}^{\infty} a_k \frac{1}{z^k}=0 
\right\}. \] 
Furthermore,
we have the convergence of
the maxima. In particular,
when $a_0$
is a complex Gaussian 
random variable, for every $t \in [1,\infty)$
we have
		\[ 
	\lim_{n \to \infty}
	\PP(\max\{|z|: P_n(z) = 0\} \leq t)	
	= \prod_{k=1}^{\infty} \left(1-t^{-2k}\right). 
		\]
\end{theorem}

\subsection{Symmetries of the models}
The jellium and the zeros of random polynomials have a lot of symmetries, which make those models particularly interesting to study. The jellium, as well as the roots of random polynomials, are stable under the inversion map $i:\mathbb C \setminus \{0\}
\to \mathbb C \setminus \{0\}$,
given by
\[i(z) = \frac{1}{z},\]
and also by scaling. This means that under these transformations, a jellium associated to $\nu$ has the same law as a jellium associated to another measure $\nu'$. The same is true for zeros of random polynomials.

\begin{theorem}[Equivariance of the jellium] \label{JelliumEquivariance}
Let $\nu$ be a probability measure satisfying \emph{Assumption} \ref{assumption1} and 
let $(x_1,\dots,x_n)$ be 
a jellium associated to $\nu$. Then, for any $\lambda >0$, $(\lambda x_1, \dots, \lambda x_n)$ has the law of a jellium associated to $\nu(\cdot/\lambda)$ and $(i(x_1),\dots,i(x_n))$ has the law of a jellium associated to $i_*\nu$, the pushforward of $\nu$ by $i$.
\end{theorem}

The pushforward $i_*\nu$ 
is well-defined since the measure $\nu$ has no atom at $0$ due to \emph{Assumption \ref{assumption1}}. 
Furthermore,  
$(i(x_1),\dots,i(x_n))$ is well-defined almost surely
because the event
$\{\exists k \text{ such that } x_k=0 \}$
has zero probability
so that the statement of Theorem \ref{JelliumEquivariance} makes sense.

\begin{theorem}[Equivariance of the zeros of random polynomials] \label{PolyEquivariance}
Let $\nu$ be a probability measure satisfying \emph{Assumption} \ref{assumption1} and let 
$z_1,\dots,z_n$ be the zeros of a random polynomial 
(given by 
 \eqref{RandomPoly})
 associated to some \text{i.i.d.} sequence
 $(a_k)_{k \in \{0,\dots,n\}}$. 
 Then, for every $\lambda >0$, 
the point process formed by
$\lambda z_1, \dots, \lambda z_n$ 
has the same law as
the one
formed by the zeros
of the random polynomial  associated to 
$\nu(\cdot/\lambda)$ and to the same sequence
$(a_k)_{k \in \{0,\dots,n\}}$. 
In addition, the point process
formed by
$i(z_1),\dots,i(z_n)$ has the same law
 as the one formed by the
 zeros of the random polynomial
 associated to $i_*\nu$, the pushforward of $\nu$ by $i$, and to the same
 sequence $(a_k)_{k \in \{0,\dots,n\}}$. 
\end{theorem}

\begin{remark}[Equivariance under M\"obius transformations]
If we allow non-radial $\nu$,
the jellium is also equivariant under translations. This implies the equivariance of the jellium with respect to all M\"obius transformations. Zeros of random polynomials with complex Gaussian coefficients are also equivariant under translations, hence under 
all M\"obius transformations. For general coefficients and non-radial measures $\nu$, the model of random polynomials depends on the choice of basis
and this basis would have to be transformed
accordingly for the equivariance to hold.
\end{remark}

The next result is standard and it is contained,
for instance, in
\cite[Section 5.4]{HoughKrisPeresVirag}
where a connection to Gaussian analytic functions 
is made.
It follows from the equivariance of the 
Bergman kernel together with
the change of variables formula in
Lemma \ref{lem:ChangeOfVariables}.

\begin{theorem}[The Bergman point process is conformally invariant]\label{BergmanEquivariance}
	Let $U_1$ and $U_2$ be open sets in $\CC$
	and suppose that 
	there exists a biholomorphism $\varphi$ 
	from $U_1$ to $U_2$. Then,
	\[ \varphi(\mathcal{B}_{U_1}) 
	\overset{\rm law}{=} \mathcal{B}_{U_2}.\]
\end{theorem}

These theorems combined together allow us to reduce
the proofs to the study of the point processes
near the origin.

\begin{remark}[On the disk
and the complement of a disk]
\label{Rem:DiskAndComplement}
Notice that the Bergman kernel
of $D_R \setminus \{0\}$
is the restriction of
the Bergman kernel of $D_R$ since
a square integrable singularity
is a removable singularity.
This implies that
the restriction of 
$\mathcal B_{D_R}$ 
to $D_R \setminus \{0\}$ is
$\mathcal B_{D_R \setminus\{0\}}$.
In particular, by applying
Theorem \ref{BergmanEquivariance},
we can say that
\[i(\mathcal B_{D_R})
	\overset{\rm law}{=}
	 \mathcal B_{
\CC \setminus \bar D_R}\]
where $i(\mathcal B_{D_R})$ makes
sense since $0 \notin 
\mathcal B_{D_R}$
almost surely or, strictly speaking,
$i(\mathcal B_{D_R})$ means  
the inverse of the restriction of
$\mathcal B_{D_R}$ to 
$D_R\setminus\{0\}$. This can also be
obtained
by using the explicit formulas given
in Definition 
\ref{def:BergmanPointProcess}.

\end{remark}

\section{Comments and perspectives}

\subsection{Related results}

A nice introduction to the theory of Coulomb
gases can be found in
\cite{SerfatyReview}.
Coulomb gases are usually studied
for strongly confining potentials.
\cite{Rider} showed
that the particles
of a Ginibre ensemble 
converge
towards the closed unit disk and 
that the particle farthest 
from zero exhibits Gumbel fluctuations.
This result
has been further
generalized and
different cases
have been found in \cite{ChafaiPeche,Seo,
JiangQi,ChangLiQi,GuiQi,Lacroix,GarciaZelada3,
ChafaiGarciaZeladaJung}, always
in the radial determinantal setting.
Very recently, \cite{Ameur} showed a control
of the distance between the particles
and the support of the equilibrium measure
for general temperatures 
and not necessarily radial potentials.

A universal limit point process at
a point  in the bulk where the equilibrium measure
has a positive density has been established,
for instance, in \cite{BermanBulk}.
Further behaviors in the bulk are
studied in \cite{AmeurSeo,AmeurKangSeo2,
GarciaZelada3}
and a nice condition has been given
for a radial case in 
Theorem \ref{BulkMinMax}
where the regularity outside
the origin is not needed.
Limiting point process
at the edge have been found in
\cite{AmeurHardEdge,AmeurKangMakarov,
HedenmalmWennman,GarciaZelada3,AmeurKangSeo}. 
To our knowledge,
this is the first time that the limiting 
behavior of the point process for 
weakly confining potentials has been
studied outside the bulk.

\subsection{Open questions}
For non-radial measures $\nu$ at inverse temperature $2$, we expect our results to generalize: for any
connected component $U$ of $\CC \setminus \text{supp }\nu$, simply connected, the outliers should converge towards the Bergman point process $\mathcal{B}_U$ in the jellium case as well as 
in the case of the zeros of random polynomials with complex Gaussian coefficients. For general coefficients, for a good choice of basis, we expect to see the zeros of a random function of the form
\[ z \mapsto \sum_{k=0}^{\infty} a_k \varphi(z)^k \]
where $\varphi$ is a conformal map from $U$ to the unit disk.

If we do not assume that the inverse temperature 
is $2$, all the results presented in this article fall. We hope that similar results hold for any inverse temperature $\beta$. In dimension one, 
the Sine point process and the Airy point process have $\beta$ counterparts which generalizes them to any temperature. 
See \cite{ValkoVirag} and  \cite{RamirezRiderVirag}.
We can dream of a generalization to any $\beta$ of the Bergman point process.

For random polynomials, the study of the outliers in the case where $\nu$ is not radial seems hard to study. One may try to 
understand the behavior of the orthogonal polynomials $R_{k,n}$ associated to an inner product
\[ \langle P, Q \rangle_{n,\nu} = 
\int P(z) \overline{Q(z)} e^{-2nV^{\nu}(z)}{\rm d}\nu(z) \]
which is in general a difficult question. 
After it is understood, 
the study of the outliers could be carried out by studying the asymptotics of the covariance kernel of the Gaussian field $(P_n(z))_{z \in \CC}$
\[ K_n(z,w)=\sum_{k=0}^n R_{k,n}(z) 
\overline{R_{k,n}(w)} \]
outside of the support of $\nu$.

\section{Proof of the equivariance results}
We start by proving 
the equivariance for the jellium, then for the zeros of random polynomials and finally for the Bergman point processes. We remark that these theorems and 
their proofs
are geometric in nature and that they can be
nicely explained by using the language of
complex line bundles on a regular setting.

\begin{proof}[Proof of Theorem \ref{JelliumEquivariance}]
	The equivariance under scaling is straightforward. We only prove the equivariance with respect to the inversion.
	Let $(x_1,\dots,x_n)$ be a jellium associated to a probability measure $\nu$ satisfying \emph{Assumption } \ref{assumption1}. This means that it follows the law
	\begin{equation}\label{gaz2}
	\frac{1}{Z_n} \exp 
	\left(-2  \left[ -\sum_{i<j} \log |x_i-x_j| + 
	(n+1) \sum_{i=1}^n V^{\nu}(x_i)  \right]
	\right) {\rm d}\ell_{\mathbb{C}^n}(x_1, \dots, x_n)
	\end{equation}
	where $Z_n$ is a normalization constant. 	
	Define 
	\[G^{V^{\nu}}(x,y) = -\log|x-y| + V^{\nu}(x) + V^{\nu}(y)\]
	and define the positive measure $\pi$ by 
	${\rm d}\pi=e^{-4V^{\nu} }{\rm d}\ell_{\mathbb{C}}$.
	Using these definitions we can write
	\begin{align}
	\exp 
	&\left(-2  \left[ -\sum_{i<j} \log |x_i-x_j| + 
	(n+1) \sum_{i=1}^n V^{\nu}(x_i)  \right]
	\right) 
	{\rm d}\ell_{\mathbb{C}^n}(x_1, \dots, x_n)	
	\nonumber	\\
	&= \exp \left(-2\sum_{i<j} G^{V^{\nu}}(x_i,x_j) \right)
	{\rm d}\pi^{\otimes_n}(x_1,\dots x_n)
	\label{eq:ExpTimesPi}.
	\end{align}
We define the function $\tilde{V}$ as\[ \tilde{V}(z) = {V^{\nu}} \left(\frac{1}{z} \right) + \log |z|.	\]
By a straightforward calculation, 
we obtain that
\[G^{V^{\nu}}(i(x),i(y))= 
-\log|x-y| + \tilde V(x) + \tilde V(y)
=:G^{\tilde V}(x,y)\]
and that
$\tilde \pi := i_*\pi $ (the pushforward measure
of $\pi$ by $i$) is given by
\[{\rm d}\tilde \pi = e^{-4\tilde V } 
{\rm d}\ell_{\mathbb{C}}.\]
In summary, the `inverse' of 
$e^{-2\sum_{i<j} G^{V^\nu}(x_i,x_j)} $ is 
$e^{-2\sum_{i<j} G^{\tilde V}(x_i,x_j)} $
and the `inverse' of $\pi$ is $\tilde \pi$
so that the inverse of
\eqref{eq:ExpTimesPi}
is
\begin{align*}
&\exp \left(-2\sum_{i<j} G^{\tilde V}(x_i,x_j)\right)
 \mathrm d
\tilde \pi^{\otimes_n}(x_1,\dots,x_n)		\\
& \quad \quad \quad =
	\exp 
	\left(-2  \left[ -\sum_{i<j} \log |x_i-x_j| + 
	(n+1) \sum_{i=1}^n \tilde V(x_i)  \right]
	\right) 
	{\rm d}\ell_{\mathbb{C}^n}(x_1, \dots, x_n).
	\end{align*}
We finish the proof of the theorem by noticing that
\[ \forall z \in \mathbb{C}\setminus \{0\}, \quad V^{i_*\nu}(z) = 
	V^{\nu}\left(\frac{1}{z}\right) 
	+ \log|z| -V^\nu(0).  \]
	Hence, $\tilde{V}$ differs from $V^{i_*\nu}$ by a constant.
	One can remove the constant 
	$
	V^\nu(0)$ from the definition of the potential as it may enter into the normalizing constant associated to this model.
\end{proof}

\begin{proof}[Proof of Theorem \ref{PolyEquivariance}]
The equivariance under scaling is a straightforward calculation. For the inversion, let $P_n$ be 
given by \eqref{RandomPoly}
	associated to $\nu$
	and to an \text{i.i.d.} sequence
	$(a_k)_{k \in \{1,\dots,n\}}$. We show that the random polynomial
	$Q_n$ defined by
	\[ Q_n(z) = z^n
	P_n \left(\frac{1}{z}\right) \]
	has the same law as the one given by $(5)$
	associated to the measure $i_*\nu$
	and to the same sequence 	
	$(a_k)_{k \in \{1,\dots,n\}}$.

	By a change of variables formula, it can be
	seen
	that the application 
	$^*:\mathbb{C}_n[X] \to \mathbb{C}_n[X]$ that to each polynomial
	$P\in \mathbb{C}_n[X]$ 
	associates
	the polynomial $P^*
	\in \mathbb{C}_n[X]$
	given by
	\[P^*(z)=z^nP(1/z)\] is an
	isometry between 
	$ \mathbb{C}_n[X]$ with the inner product
	defined by
	\[
	\langle P,Q \rangle_{n,\nu}
	= \int_\mathbb{C} P \overline{Q} e^{-2n V^{\nu}}
	{\rm d}\nu\]
	and	$ \mathbb{C}_n[X]$ with the inner product
	defined by
	\[
	\langle P,Q \rangle_{n,i_*\nu}
	=\int_\mathbb{C} P \overline{Q} 
	e^{-2n V^{i_*\nu}} d i_*\nu.\]
	By also noticing
	that $*$ preserves the monomials,
	needed only in the general non-Gaussian case,
	the proof is completed.		
\end{proof}

\begin{proof}[Proof of Theorem \ref{BergmanEquivariance}]
We start by recalling an important 
relation satisfied by Bergman kernels on 
different open sets.
Let $U_1$ and $U_2$ be two open sets and $\varphi$ be a biholomorphism from $U_1$ to $U_2$. Then, if $K_{U_1}$
 is the Bergman kernel of $U_1$ and
 $K_{U_2}$ is the Bergman kernel of $U_2$, we have, by
 \cite[Theorem 16.5]{Bell}, 
\[ \forall z ,w \in U_1 \quad K_{U_1}(z,w) = 
\varphi'(z) K_{U_2}(\varphi(z),\varphi(w)) \overline{\varphi'(w)}  . \]
We conclude by an application of the change of 
variables formula for determinantal point processes given in 
Lemma \nolinebreak \ref{lem:ChangeOfVariables}.

\end{proof}

\section{Proof of the main results}
We start by proving a key lemma which will be essential in several of the proofs. It gives a very tractable formula for the potential of radial measures.
\begin{lemma}[Useful formula for the potential] \label{Formula potential}
	Let $\nu \in \mathcal P(\mathbb{C})$
	be a rotationally invariant probability measure
	such that
	$\int_{\CC \setminus \mathbb D} \log |x| {\rm d}\nu(x) < \infty$.
	Then 
	\[V^{\nu}(z) =
	\int_1^{|z|} \frac{\nu(D_r)}{r} {\rm d}r
	+ \int_{\CC \setminus \mathbb D}
	\log |x| {\rm d}\nu(x)
	\]
	where $D_r$ denotes the open disk of radius $r$.
\end{lemma}
This formula is 
standard and can be seen
to be related to
Poisson-Jensen formula \cite[Theorem II.4.10]{SaffTotik}.

\begin{remark}[The potential is defined up to a constant.] \label{V=0 inside}
	We recall that 
	one can choose to add a constant to the potential $V^{\nu}$ without changing the law \eqref{gaz}. Adding a constant will only change the normalizing constant $Z_N$. The potential $V^{\nu}$ 
	can be modified so that
	it is equal to 
	zero at the unit circle and then
	\[ \forall z \in \mathbb{C} \quad V^{\nu}(z) = \int_1^{|z|} \frac{\nu(D_r)}{r} {\rm d}r. \]
	In fact, $V^\nu$ defined in this way
	satisfies
	Poisson's equation 
	with source $2\pi\nu$ even if 
	the actual logarithmic
	potential \eqref{eq:potential} does not make sense. 
	Nevertheless, it is only
	when the condition of 
	Lemma \ref{Formula potential}
	is satisfied that $V^\nu$
	satisfies condition \eqref{potential}.
	We emphasize 
	again that this representation of the potential will be very helpful in the rest of the article.
\end{remark}

\begin{proof}[Proof of Lemma \ref{Formula potential}]
	
	Since $\nu$ is radial, the disintegration theorem\footnote{Conditional expectation in probabilist language.} \cite[Theorem 5.3.1]{AmbrosioGigliSavare} allows us to write
	\[\nu = \int_0^\infty l_r {\rm d}\mu(r)\]
	where $l_r \in \mathcal P(\mathbb{C})$ is the uniform
	probability measure on $C_r$, the circle centered
	at $0$ of
	radius $r$, and
	$\mu$ is a probability measure on $[0,\infty)$ characterized by
	\[ \forall a ,b >0 \quad \mu((a,b)) = \nu( \{ a<|z|<b \}  ) .\]
	This decomposition of the measure means that for any $f$ positive measurable function 
	or integrable with respect to $\nu$ we have
	
	\[ \int_\mathbb{C} f(z) {\rm d}\nu(z) = \int_0^\infty 
	\left(\int_{C_r} f(w)dl_r(w) 
	\right){\rm d}\mu(r).   \]
	Using this relation to compute the potential of the measure $\nu$
	we have that
	for every $z\in \mathbb{C}$
	\begin{align*}
	V^{\nu}(z)= \int_{\mathbb{C}}
	\log|z-w|{\rm d}\nu(w) =& \int_0^{\infty} \left(
	\int_{C_r} \log |z-w| dl_r (w)
	\right){\rm d}\mu(r) \\ =&
	\int_0^\infty V^{l_r}(z) {\rm d}\mu(r).
	\end{align*}
	where $V^{l_r}$ is the potential of the uniform measure on $C_r$. In fact, since the integrability
	of $\log|z-\cdot|$ is not yet known 
	we may proceed by a limiting argument
	by first integrating over the complement
	of an open annulus that contains $z$.
	But, since $V^{l_r}$ can be 
	computed explicitly and is equal to \cite[p.29]{Ransford}
	\[V^{l_r}(z) = \begin{cases}
	\log r 	&\text{ if }	 	|z| \leq r		\\
	\log |z|	&\text{ if }		|z| > r	
	\end{cases} \]
	we are able to complete
	the limiting argument.		
	Hence we obtain
	\begin{align}
	V^{\nu}(z) &= \int_{[|z|,\infty)}
	\log r \, {\rm d}\mu (r)
	+  \int_{[0,|z|)}\log|z|  
	{\rm d}\mu(r) 				\nonumber			\\
	&= \int_{[|z|,\infty)}
	\log r \, {\rm d}\mu (r)
	+ \log|z| \nu(D_{|z|})
	\label{eq:firstpotential}
	\end{align}
	where the last term is not there
	if $z=0$.
	Notice that,
	by \eqref{eq:firstpotential},
	the lemma is already proven
	for $|z|=1$.
	Suppose that $z \neq 0$.
	 Let us notice that Fubini's theorem implies 
	\begin{align*}
	\int_{|z|}^\infty \frac{1}{r} 
	\nu(\hspace{0.5pt}  \mathbb C \setminus D_r) {\rm d}r
	& = 
	\int_{\mathbb R^+}
	1_{r\geq  |z|}  \frac{1}{r}
	\left(\int_{\mathbb R^+} 1_{s\geq r}  {\rm d}\mu(s) 
	\right){\rm d}r 
	\\& 
	= \int_{\mathbb R^+} 1_{s \geq |z|} 
	\left(\int_{|z|}^s \frac{1}{r} {\rm d}r \right)\, 
	{\rm d}\mu(s)
	\\ &
	= \int_{[|z|,\infty)} \log s \,  
	{\rm d}\mu(s) - \log|z| \nu(\hspace{0.5pt} \CC \setminus 
	D_{|z|}).
	\end{align*}
	Then, by replacing this equality in
	\eqref{eq:firstpotential}, we obtain
	\begin{align*}
	V^{\nu}(z) &= \log|z| \left[
	\,
	\nu(D_{|z|}) + \nu(\hspace{0.5pt} \CC \setminus D_{|z|})
	\,
	\right]
	+
	\int_{|z|}^\infty 
	\frac{\nu(\hspace{0.5pt} \CC \setminus D_r)}{r} {\rm d}r	\\
	&=	\log|z|
	+
	\int_{|z|}^1 \frac{\nu(\hspace{0.5pt} \CC \setminus D_r)}{r} {\rm d}r
	+
	\int_1^\infty \frac{\nu(\hspace{0.5pt} \CC \setminus D_r)}{r} {\rm d}r	\\
	&=	\log|z|
	+
	\int_{|z|}^1 \frac{1-\nu(D_r)}{r} {\rm d}r
	+
	\int_1^\infty \frac{\nu(\hspace{0.5pt} \CC \setminus D_r)}{r} {\rm d}r	\\
	&=	\int_1^{|z|} \frac{\nu(D_r)}{r}{\rm d}r
	+
	\int_1^\infty \frac{\nu(\hspace{0.5pt} \CC \setminus D_r)}{r} {\rm d}r	\\
	&=	\int_1^{|z|} \frac{\nu(D_r)}{r}{\rm d}r
	+
	\int_{[1,\infty)} \log r \, {\rm d}\mu(r)
	\end{align*}
	where the last equality is obtained
	by taking $|z|=1$. The case $z=0$ 
	follows the same argument.
	
\end{proof}
\subsection{Results for the jellium}
\subsubsection{Proof of Theorem \ref{JelliumThm}}
\begin{proof}[Proof of Theorem \ref{JelliumThm}]
Due to the equivariance of the jellium, stated in Theorem \ref{JelliumEquivariance}, and the equivariance of the Bergman point process 
from Theorem \eqref{BergmanEquivariance}
together with 
Remark \ref{Rem:DiskAndComplement}, it suffices to assume that $U$, the connected component of $\CC \setminus \text{supp }\nu $, is the open unit disk, which can be written as 
\[S^1 \subset \mathrm{supp}\, \nu \subset 
\CC \setminus \mathbb{D}.\]
	
\newpage	
	
\textbf{Step 1: Kernel of the Coulomb gas}
	
	\noindent
	Let $\nu$ be a measure satisfying 
	\emph{Assumption \ref{assumption1}} 
	such that $S^1 \subset \mathrm{supp}\, \nu \subset \CC \setminus \mathbb{D}$. 
	Let $(x_1,\dots,x_n)$ be a jellium associated to $\nu$. Notice that, due to Lemma \ref{Formula potential}, $V^{\nu}$ is constant in the unit disk, which we set to be equal to $0$
	(see Remark \ref{V=0 inside}). The point process $\{x_1,\dots,x_n  \}$ is determinantal, because it is a Coulomb gas in the plane with inverse temperature $2$. It is associated to the kernel
	$K_n:\mathbb C \times \mathbb C \to \mathbb C $
	defined by
	\[K_n(z,w) = \sum_{k=0}^{n-1} b_{k,n} 
	z^k \bar w^k e^{-(n+1)V^{\nu}(z)} 
	e^{-(n+1)V^{\nu}(w)}\]
	where
	\[(b_{k,n})^{-1}
	=\int_\mathbb C |z|^{2k} e^{-2(n+1)V^{\nu}(z)} 
	{\rm d}\ell_{\mathbb{C}}(z).  \]
	The point process $\mathcal{I}_n=\{x_k : |x_k|<1  \}$ is the restriction to the open unit disk of the point process $\{x_1,\dots,x_n \}$, and is also a determinantal point process, with kernel given by the restriction of $K_n$
to $\mathbb D\times \mathbb D$. We will also denote this kernel
by $K_n$. 
By \cite[Proposition 3.10]{ShiraiTakahashi},
stated in Proposition \ref{prop:ConvergenceOfDPP} 
for convenience, in order to prove the convergence of the sequence of point processes $(\mathcal{I}_n)_{n \in \NN^+}$ towards the Bergman point process
of $\mathbb D$, it is enough to prove that the sequence of kernels $(K_n)_{n \in \NN^+}$ converges uniformly on compact subsets of $\mathbb D \times \mathbb D$ towards
	\[ K(z,w) = \frac{1}{\pi} \frac{1}{(1-z\bar{w})^2} = \sum_{k=0}^{\infty} \frac{k+1}{\pi}z^k \bar{w}^k. \]
	In fact, the proof will work,
	and thus the theorem is true, as soon as 
	the radial potential
	$V$ satisfying \eqref{potential}
	\emph{is zero inside of the closed unit disk
		and positive outside of it}. 
	
	\textbf{Step 2: Convergence of the
		coefficients}
	
	\noindent
	First, let us notice that 
	\begin{align*}
	(b_{k,n})^{-1}
	&=\int_\mathbb C |z|^{2k} e^{-2(n+1)V^{\nu}(z)} 
	{\rm d}\ell_{\mathbb{C}}(z)
	=2\pi\int_0^\infty r^{2k+1} e^{-2(n+1)V^{\nu}(r)} 
	{\rm d}r										\\
	&=
	2\pi\int_0^1 r^{2k+1} e^{-2(n+1)V^{\nu}(r)} 
	{\rm d}r
	+
	2\pi\int_1^\infty r^{2k+1} e^{-2(n+1)V^{\nu}(r)} 
	{\rm d}r,
	\end{align*}
	where we use $V^\nu(r)$ to denote
	$V^\nu$ evaluated at any point
	of norm $r$. Since the potential $V^{\nu}$ is equal to $0$ inside the unit disk, one can compute the first term
	\[2\pi\int_0^1 r^{2k+1} e^{-2(n+1)V(r)} 
	{\rm d}r=
	2\pi\int_0^1 r^{2k+1} 
	{\rm d}r
	=\frac{\pi}{k+1}.
	\]
	For the second term, let us prove that
	\[\lim_{n \to \infty}
	\int_1^\infty r^{2k+1} e^{-2(n+1)V^{\nu}(r)}
	 {\rm d}r
	= 0.\]
	But this is a consequence of Lebesgue's dominated
	convergence theorem where we use the bound
	\[r^{2k+1} e^{-2(n+1)V^{\nu}(r)}
	\leq r^{2k+1} e^{-2(k+2)V^{\nu}(r)}\]
	for $k \leq n-1$. The fact that
	$r^{2k+1} e^{-2(n+1)V^{\nu}(r)}$ goes
	to zero when $r >1$ can be seen from the
	fact that $V^{\nu}(r)>0$ 
	which in turn can be seen
	from the formula in Lemma \ref{Formula potential} as
	follows. Let $r>1$ and write
	\[V^{\nu}(r) =
	\int_1^{r} \frac{\nu(D_s)}{s}{\rm d}s.\]
	If $V^{\nu}(r)$ were zero the integrand 
	$\nu(D_s)$ would be zero for almost every
	$s \in [1,r]$ which is impossible because
	$\nu(D_s)>0$ for $s>1$.
	
	\noindent In summary, we obtain
	$\lim_{n \to \infty} b_{k,n}
	= \frac{k+1}{\pi}.$
	
	\textbf{Step 3: Convergence
		of the kernels}
	
	\noindent
	Let us fix $\rho \in (0,1)$ then for any $z$, $w$ inside the disk of radius $\rho$ we have
	\[
	|K_n(z,w) - K(z,w)|  \leq \sum_{k=0}^{\infty} 
	\left|b_{k,n} - \frac{k+1}{\pi}\right| |z|^k |w|^k\leq\sum_{k=0}^{\infty} 
	\left|b_{k,n} - \frac{k+1}{\pi}\right|
	\rho^{2k} . \]
	The right-hand term converges to zero as $n$ goes to infinity by an application of Lebesgue's dominated convergence theorem, noticing that
	\[ \forall k ,n \quad 0 \leq b_{k,n} \leq \frac{k+1}{\pi}.  \]
	which implies
	\[ \left|b_{k,n} - \frac{k+1}{\pi}
	\right|\rho^{2k} \leq 
	2\left( \frac{k+1}{\pi} \right)\rho^{2k}. \]
		By \cite[Proposition 3.10]{ShiraiTakahashi} the point process
	$\mathcal{I}_n$
	converges
	towards the Bergman point process $\mathcal B_{\mathbb{D}}$.
	\end{proof}

\begin{proof}[Proof of Corollary \ref{CorollaireJellium}]
In the case where $U$ is the open unit disk, the point process of the outliers $\mathcal{I}_n$ converges towards
	$\mathcal B_{\mathbb{D}}$.
	By the continuity
	of the minimum (Lemma \ref{lemma:min} in the appendix)
	we obtain
	that 
	the minimum of the norms of $\mathcal{I}_n $
	converges to
	the minimum of the norms of $B_{\mathbb{D}}$.
	But the limit of the minimum
	of the norms of $\mathcal{I}_n $
	coincides with the limit of the minimum of
	$\{|x_1|,\dots,|x_n|\}$
	since the latter
	limit is bounded by $1$.
	
	Thanks to \cite[Theorem 4.7.1]{HoughKrisPeresVirag}, the set of norms 
of the Bergman point process 
of $\mathbb D$ has the same law as $\{ U_k^{1/2k}, k \in \NN 
	\setminus\{0\} \}$, with the ${U_k}'s$ being independent uniform random variables on $[0,1)$. This immediately implies that 
		\[\PP\left(\, \inf\{|z| : z\in
	\mathcal {B}_{\mathbb D}\}
	\leq t \right)=
	1-\prod_{k=0}^{\infty} (1-t^{2k}).
	\]
We conclude the proof by applying the inversion and a scaling.
\end{proof}

\subsubsection{Proof of Theorem \ref{BulkMinMax}}
\begin{proof}
We prove the first part of the theorem, and we deduce the rest thanks to the equivariance under inversion of the jellium.

\noindent
	\textbf{Proof at the origin.}
	
	\textbf{Step 1: Kernel of the rescaled Coulomb gas}
	
	\noindent Let $(x_1,\dots,x_n)$ be a
	jellium associated to $\nu$. 
	Then the point process $\left\{ \frac{x_1}{n^{1/\alpha}} , \dots, \frac{x_n}{n^{1/\alpha}}  \right\}$ is a determinantal point process associated to the kernel
	\begin{align*}
	\tilde K_n(z,w) & = 
	\frac{1}{n^{2/\alpha}} 
	K_n\left( \frac{z}{n^{1/\alpha}}, 
	\frac{w}{n^{1/\alpha}}\right) \\
	& = \sum_{k=0}^n b_{k,n} z^k \bar{w}^k 
	e^{-(n+1)V^{\nu}
		\left(\frac{z}{n^{1/\alpha}}\right)}
	e^{-(n+1)V^{\nu}
		\left(\frac{w}{n^{1/\alpha}}\right)} 
	\end{align*}
	where
	\[ b_{k,n}^{-1} = 
	2 \pi \int_0^\infty r^{2k+1} e^{-2(n+1)V^{\nu}
		\left(\frac{r}{n^{1/\alpha}} \right)} 
		{\rm d}r. \]
	This may be seen, for instance, 
	by the change of variables formula
	in Lemma $\ref{lem:ChangeOfVariables}$.
	We will prove that the sequence of kernels 
	$(\tilde K_n)_{n \in \NN}$ converges uniformly 
	on compact subsets of $\mathbb{C} \times \mathbb{C}$ towards the kernel
	\begin{equation}
	\label{eq:K}
	K(z,w)= \sum_{k=0}^{\infty} 
	b_k z^{k} \bar{w}^k 
	e^{-\gamma |z|^{\alpha}}e^{-\gamma |w|^{\alpha}} 
	\end{equation}	 
	where
	\[\gamma=\frac{\lambda}{\alpha}
	\, \, \mbox{ and } \, \, b_k^{-1} = 
	2 \pi \int_0^\infty 
	r^{2k+1} e^{-2\gamma r^{\alpha}} {\rm d}r.  \]
	To prove this convergence, we will first prove that $\lim_{n \to \infty} b_{k,n} =b_k$, then we will find a sequence $(B_k)_{k \in \NN}$ such that $\sum_{k=0}^\infty B_k r^k$ 
has an infinite radius of convergence and $b_{k,n} \leq B_k$ for every 
	\nolinebreak $n$. This will imply the uniform convergence of the kernels.
	
	\textbf{Step 2: Properties
		satisfied by the potential}
	
	\noindent Let $\nu$ be a rotationally invariant probability measure such that there exists $\alpha >0$ and 
	$\lambda >0$ with
	\[ \lim_{r \to 0} \frac{\nu(D_r)}{r^{\alpha}} = \lambda.  \]
	Since the potential of $\nu$ can be written as
	\[ V^{\nu}(z) = \int_{1}^{|z|} \frac{\nu(D_r)}{r} {\rm d}r, \]
	we obtain that 
	\[ \frac{ V^{\nu}(r) - V^{\nu}(0)}{r^{\alpha}} \xrightarrow[r \to 0]{} \frac{\lambda}{\alpha}.  \]
	and that $V^{\nu}(r) > V^{\nu}(0)$ for every
	$r > 0$.
	From now on, we will assume that $V^{\nu}(0)=
	\nolinebreak 0$, 
	since adding a constant to the potential $V^{\nu}$ does not change the law \eqref{gaz}. 
	Using this new convention, we have
	\begin{equation}
	\label{eq: potential at zero}
	\lim_{r \to 0} \frac{V^{\nu}(r)}{r^{\alpha}} = \frac{\lambda}{\alpha} =: \gamma 
	\end{equation}
	and 
	\begin{equation}
	\label{eq: positive potential}
	V^\nu(r) > 0 \mbox{ for every } r>0.
	\end{equation}
	In fact, those two properties of the potential
	are the only properties needed, apart from
	\eqref{potential},
	for the theorem
	to be true.

	\textbf{Step 3: Convergence of the
		coefficients}
	
	\noindent We prove that $b_{k,n}$ converges to
	$b_k$ as $n$ goes to infinity
	or, equivalently,
	\[\lim_{n \to \infty}
	\int_0^\infty r^{2k+1} e^{-2(n+1)V^{\nu}
		\left(\frac{r}{n^{1/\alpha}} \right)} {\rm d}r
	=\int_0^\infty 
	r^{2k+1} e^{-2\gamma r^{\alpha}} {\rm d}r
	\]
	We divide the integral in
	three parts.
	\begin{align*} 
	&\int_0^\infty  r^{2k+1} 
	e^{-2(n+1) V^{\nu}
		\left( \frac{r}{n^{1/\alpha}} \right)}
	{\rm d}r \\
	&=	\int_0^{n^{1/\alpha}	\varepsilon } r^{2k+1} 
	e^{-2(n+1) V^{\nu}\left( \frac{r}
		{n^{1/\alpha}} \right)} {\rm d}r			\\
	&+
	\int_{n^{1/\alpha} \varepsilon  }
	^{n^{1/\alpha} M } r^{2k+1} 
	e^{-2(n+1) V^{\nu}\left( \frac{r}{
			n^{1/\alpha}} \right)} {\rm d}r			\\
	&+
	\int_{n^{1/\alpha} M }^\infty r^{2k+1} 
	e^{-2(n+1) 
		V^{\nu}\left( \frac{r}{
			n^{1/\alpha}} \right)} {\rm d}r	
	\end{align*}
	where we have chosen 
	$\varepsilon > 0$
	such that
	$\frac{\gamma}{2} r^\alpha \leq V^{\nu}(r)$
	for $|r| \leq \varepsilon$
	and 
	$M > \varepsilon$ such that
	$\frac{1}{2}\log|r| \leq V^{\nu}(r)$
	for $|r| \geq M$.
	
	We also know, by the continuity
	and the positivity outside $0$ of $V^{\nu}$ that
	there exists a constant $C>0$ such that
	$C \leq V^{\nu}(r)$ for $r \in [\varepsilon, M]$.
	
	Since \[e^{-2(n+1) V^{\nu}\left( \frac{r}
		{n^{1/\alpha}} \right)} 
	1_{[0,n^{1/\alpha} \varepsilon 	]}(r)
	\leq e^{-\frac{(n+1)}{n} \gamma r^\alpha }
	\leq e^{- \gamma r^\alpha },\]
	we can
	use Lebesgue's dominated convergence theorem 
	for the first term. The second term is bounded by
	\[\int_{n^{1/\alpha} \varepsilon }
	^{n^{1/\alpha} M } r^{2k+1} 
	e^{-2(n+1) C} {\rm d}r\]
	which goes exponentially fast to zero when
	$n \to \infty$.	
	
	The last integral is bounded by
	\begin{align*}
	\int_{n^{1/\alpha} M}^\infty r^{2k+1} 
	e^{-2(n+1)  \log \left( \frac{r}
		{n^{1/\alpha}}\right)} {\rm d}r
	& =n^{(2k+2)/\alpha}
	\int_{M}^\infty \rho^{2k+1} 
	e^{-2(n+1)  \log \rho } {\rm d}\rho \\	
	& =n^{(2k+2)/\alpha}
	\int_{M}^\infty \rho^{2k - 2n -1} {\rm d}\rho 
	\xrightarrow[n \to \infty]{} 0.
	\end{align*}

	\textbf{Step 4: Convergence
		of the kernels}
	
	\noindent Notice that, uniformly on compact sets,
	\[ (n+1) V^{\nu} \left(\frac{|z|}{n^{1/\alpha}} \right)  \xrightarrow[n \to \infty]{}\gamma |z|^{\alpha} 
	\]
	due to \eqref{eq: potential at zero}.
	Then, it is left to prove that
	\[\lim_{n \to \infty}
	\sum_{k=0}^n b_{k,n} z^k \bar{w}^k 
	= 
	\sum_{k=0}^{\infty} b_k z^{k} \bar{w}^k 
	\]
	uniformly on compact sets
	of $\mathbb{C} \times \mathbb{C}$.
	
	Take
	$\varepsilon>0$ such that 
	$2\lambda r^\alpha \geq V(r)$
	for $|r| \leq \varepsilon$.	 Then,
	for every positive integers $k,n$ such that
	$k \leq n-1$,
	\begin{align*}
	\int_0^\infty  r^{2k+1} 
	e^{-2(n+1) V^{\nu}\left( \frac{r}{n^{1/\alpha}}
		\right)} {\rm d}r 
	&\geq
	\int_0^{n^{1/\alpha} \varepsilon} r^{2k+1} 
	e^{-4\frac{n+1}{n} \gamma r^\alpha} {\rm d}r				\\
	&\geq
	\int_0^{n^{1/\alpha} \varepsilon } r^{2k+1} 
	e^{-8 \gamma r^\alpha} {\rm d}r								\\
	&\geq
	\int_0^{k^{1/\alpha} \varepsilon } r^{2k+1} 
	e^{-8 \gamma r^\alpha} {\rm d}r.
	\end{align*}
	This suggests
	us to define $B_k$ by
	\[(B_k)^{-1} = 
	\int_0^{k^{1/\alpha} \varepsilon } r^{2k+1} 
	e^{-8 \gamma r^\alpha} {\rm d}r
	=k^{(2k+2)/\alpha}
	\int_0^{\varepsilon } \rho^{2k+1} 
	e^{-8 \gamma k \rho^\alpha} \mathrm d
	\rho	
	=k^{(2k+2)/\alpha}
	\int_0^{\varepsilon } 
	e^{k \left( 2\log \rho -8\gamma \rho^\alpha
		\right)}
	\rho \, \mathrm d
	\rho	.\]
	By the root test,
	for $\sum_{k=0}^\infty B_k x^k$ to
	converge for every $x>0$, we need
	that
	\[\lim_{k \to \infty}
	\frac{1}{k}\log \left[ (B_k)^{-1} \right]
	=\infty.\] We know that
	$\lim_{k \to \infty} \frac{1}{k} \log 
	\left[ k^{(2k+2)/\alpha} \right]
	= \infty$ so that it would be enough to 
	prove that
	$\frac{1}{k}\log\int_0^{\varepsilon } 
	e^{k \left( 2\log \rho -8\gamma \rho^\alpha
		\right)}
	\rho\, \mathrm d
	\rho	$
	is bounded from below. In fact,
	by the Laplace's method we know that
	\[\lim_{k \to \infty}
	\frac{1}{k}\log\int_0^{\varepsilon } 
	e^{k \left( 2\log \rho -8\gamma \rho^\alpha
		\right)}
	\rho\,  \mathrm d
	\rho	
	=
	\sup_{\rho \in [0,\varepsilon]}\{2 \log \rho - 
	8 \gamma\rho^{\alpha} \}
	> -\infty.
	\]
	Take $R > 0$ and suppose $|z|, |w| \leq R$.	 We have
	\[\left|\sum_{k=0}^{n-1} b_{k,n}
	z^k \bar w^k
	-
	\sum_{k=0}^\infty b_k z^k \bar w^k
	\right|
	\leq
	\sum_{k=0}^{\infty} |b_{k,n} - b_k|
	|z|^k |\bar w|^k
	\leq
	\sum_{k=0}^{\infty} |b_{k,n} - b_k|
	R^{2k},
	\]
	where we have defined $b_{k,n} = 0$
	for $k \geq n$.
	Since $|b_{k,n} - b_k|R^{2k}$
	is bounded by
	$2B_k R^{2k}$ 
	we can use Lebesgue's dominated convergence
	theorem to conclude.
	
	\textbf{Step 5: Convergence of the point
		process and the minima}
	
	\noindent
	By \cite[Proposition 3.10]{ShiraiTakahashi}, 
	stated in Proposition
	\ref{prop:ConvergenceOfDPP}, the point process
	$\left\{ \frac{x_1}{n^{1/\alpha}} , \dots, \frac{x_n}{n^{1/\alpha}}  \right\}$
	converges
	to a determinantal point process $P$
	associated to the kernel $K$ defined in
	\eqref{eq:K}. By the continuity
	of the minimum (Lemma \ref{lemma:min})
	we obtain
	that 
	$\left\{ \frac{|x_1|}{n^{1/\alpha}} , \dots, 
	\frac{|x_n|}{n^{1/\alpha}}  \right\}$
	converges in law to
	the minimum of the norms of $P$.

	\textbf{Step 6: Analysis of the limit 
	of the minima}	
	
	\noindent 
	Let $\{Y_k\}_{k \geq 0}$ be a sequence
	of positive independent 
	random variables
	such that $Y_k$ follows the law
	\[\frac{	r^{2k+1} e^{-2\gamma r^{\alpha}}{\rm d}r}
	{\int_0^\infty s^{2k+1} 
		e^{-2\gamma s^{\alpha}}{\rm d}s}.\]
	If 
	$P$ is the determinantal point process
	associated to the kernel $K$
	then, by
	\cite[Theorem 4.7.1]{HoughKrisPeresVirag},
	the law of
	$\{|z| : z \in P\}$
	is the same as the law of
	the point process defined by
	$\{Y_k\}_{k \geq 0}$.
	So the infimum 
	has cumulative distribution function
	\begin{align*}
	\PP\left(\inf\{Y_k\} \leq y \right) 
	=
	1 - \PP\left(\inf\{Y_k\} > y \right) 
	&=
	1 - \prod_{k=0}^\infty
	\PP\left(Y_k > y \right) 
	\end{align*}
	But, by a change of variables we may see that
	\begin{align*}
	\int_y^\infty 
	r^{2k+1} e^{-2\gamma r^{\alpha}}{\rm d}r
	& =\frac{1}{\alpha(2\gamma)^{(2k+2)/\alpha}}
	\int_{2\gamma y^\alpha}^\infty
	\rho^{(2k+2)/\alpha - 1} e^{-\rho} {\rm d}\rho \\
	& =\frac{1}{\alpha(2\gamma)^{(2k+2)/\alpha}}
	\Gamma
	\left(\frac{2k+2}{\alpha},2\gamma y^\alpha
	\right)
	\end{align*}
	so that
	\[\PP\left(Y_k > y \right)
	=\frac{	\int_y^\infty 
		r^{2k+1} e^{-2\gamma r^{\alpha}}{\rm d}r}
	{\int_0^\infty s^{2k+1} 
		e^{-2\gamma s^{\alpha}}{\rm d}s}
	=
	\frac{\Gamma
		\left(\frac{2k+2}{\alpha},2\gamma y^\alpha
		\right)}
	{\Gamma
		\left(\frac{2k+2}{\alpha} \right)}\]
	from which we have that
	\begin{equation}
	\label{eq:FormulaInf}	
	\PP\left(\inf\{Y_k\} \leq y \right) 
	=
	1-
	\prod_{k=0}^\infty
	\frac{\Gamma
		\left(\frac{2k+2}{\alpha},2\gamma y^\alpha
		\right)}
	{\Gamma
		\left(\frac{2k+2}{\alpha} \right)}
	\end{equation}	
	
	{\bf Proof at infinity and convergence
	of the maxima.} 
	
To prove the result at infinity, we use the equivariance of the jellium under inversion
since, if $\nu$ satisfies 
	$\lim_{r \to \infty} r^{\alpha} 
	\nu(\hspace{0.5pt} \CC \setminus D_r) 	
	= \lambda$, then its pushforward by the inversion, $i_*\nu$, satisfies $\lim_{r \to 0} i_*\nu(D_r)/r^{\alpha} = \lambda$. To prove the result
	about the maxima, 
	we use the same equivariance under inversion
	together with the convergence of
	the minima from Step 5 and 
	the cumulative distribution function from
	\eqref{eq:FormulaInf}.
	\end{proof}

\subsection{Results on random polynomials}
We prove the results on random polynomials in the same order as we did for the jellium: We prove Theorem \ref{PolyThm} 
and Corollary
\ref{CorollairePoly} using the same strategy, and later 
we prove Theorem \ref{BulkMinMaxPoly}.
At the end of this section
we
give a short proof
of Theorem \ref{th:Weyl}.

First, we recall that if $(a_k)_{k \in \NN}$ are i.i.d. random variables satisfying 
\[\EE(\log(1+|a_0|))<+\infty,\] 
then the random power series $\sum_{k=0}^{\infty} a_k z^k$ has almost surely a radius of convergence equal to one. In fact, the following lemma immediately implies the general
statement in Corollary \ref{cor:ConvergenceRadius}
below.

\begin{lemma}[Arnold \cite{Arnold}] \label{Chap3 lemma}
	Let $(a_k)_{k\in \NN}$ be 
	a sequence of i.i.d. complex random variables. 
	Fix $\varepsilon > 0$. Then 
	\[\sup\limits_{k\in \NN} \frac{|a_k|}{e^{\varepsilon k}} < \infty \   \mathrm{  a.s.}   \Longleftrightarrow \EE( \log (1+|a_0|)) < \infty.  \]
\end{lemma}

\begin{proof}[Proof of the lemma]
	For every non negative random variable $X$ we have:
	\[ \sum\limits_{k=1}^{\infty} \PP\left(X\geq k\right) \leq \EE(X) \leq \sum\limits_{k=0}^{\infty} \PP(X\geq k) . \]
	Those inequalities come from the relation: $\EE(X)= \int_{\RR^+} \PP(X\geq x) \mathrm d x$. Now we apply this inequality to the non-negative random variable \[X= \frac{1}{\varepsilon} \log (1+|a_0|).\] 
	We deduce that 
	\[\sum\limits_{k=1}^{\infty} 
	\PP\left(\frac{1+|a_k|}{e^{\varepsilon k}}
	 \geq 1\right) < \infty
	 \Longleftrightarrow
	 \EE\left(\log (1+|a_0|) \right) < \infty.\]
Borel-Cantelli lemma 
 implies that 
\[\mbox{a.s. } \,
\frac{1+|a_k|}{e^{\varepsilon k}}
	 \geq 1 \mbox{ for a finite number of }
	 k \Longleftrightarrow 
	 \EE\left(\log (1+|a_0|) \right) < \infty.\]
In particular, if 	 
$\EE\left(\log (1+|a_0|) \right) < \infty$ then	 
$\varlimsup_{k \to \infty} \frac{|a_k|}{e^{\varepsilon k}} \leq 1,$ which implies $\sup\limits_{k\in \NN} \frac{|a_k|}{e^{\varepsilon k}} < \infty$.
Conversely, if $\sup\limits_{k\in \NN} \frac{|a_k|}{e^{\varepsilon k}} <\infty$ 
then, for $\theta >\varepsilon$
we have
$\lim_{k \to \infty}
\frac{|a_k|}{e^{\theta k}} = 0 $
almost surely so that
$\frac{1+|a_k|}{e^{\theta k}}
	 \geq 1$ for a finite number of $k$
	 which implies that
	 $\EE\left(\log (1+|a_0|) \right) < \infty$.
\end{proof}

\begin{corollary}[Radius of convergence
of a random power series]
\label{cor:ConvergenceRadius}
Let $(\eta_k)_{k \in \mathbb N}$
be a (deterministic) sequence
of complex numbers.
Suppose that $(a_k)_{k \in \mathbb N}$ 
is a sequence of i.i.d. complex
random variables such that
\[\EE( \log (1+|a_0|)) < \infty.\]
If $a_0$ is not zero (i.e. the law
of $a_0$ is not the Dirac delta at $0$)
then the radius of convergence
of the random power series 
$\sum_{k=0}^{\infty} \eta_k a_k z^k$ 
is almost surely equal to
the radius of convergence of the deterministic
power series
$\sum_{k=0}^{\infty} \eta_k z^k$. 
\begin{proof}

Lemma \ref{Chap3 lemma} implies that
\[\varlimsup_{k \to \infty} {|a_k|^{1/k}}
\leq 1 \, \text{ a.s.}
\Longleftrightarrow
\EE( \log (1+|a_0|)) < \infty.
\]
Since
\[\varlimsup_{k \to \infty} {|\eta_k|^{1/k}}
\varlimsup_{k \to \infty} {|a_k|^{1/k}}
\geq
\varlimsup_{k \to \infty} {|\eta_k a_k|^{1/k}}
\]
we obtain that the radius of convergence of 
$\sum_{k=0}^{\infty} \eta_k a_k z^k$ 
is greater or equal than
the radius of convergence of
$\sum_{k=0}^{\infty} \eta_k z^k$.
On 
the other hand,
if $\eta_k z^k$ is not bounded,
then
$\eta_k a_k z^k$ is not bounded. This is a 
consequence of the second
Borel-Cantelli lemma since
we can find $\varepsilon>0$ such that
\[\mathbb P(|a_0|> \varepsilon) > 0\]
which implies that, almost surely,
\[|a_k|> \varepsilon
\text{ for infinitely many } k.\]
This concludes the proof.

\end{proof}

\end{corollary}
\subsubsection{Proof of Theorem \ref{PolyThm}}
\begin{proof}[Proof of Theorem \ref{PolyThm}]
	{\bf Premilinaries.}
	Thanks to the equivariance of the roots of random polynomials
	given
	in Theorem
	\ref{PolyEquivariance}, 
	we can assume that $U$, the connected component of $\CC \setminus \text{supp }\nu$, 
	is the open unit disk. This condition can be written as $S^1 \subset \text{ supp }\nu \subset \CC\setminus \mathbb{D}$.
	
	{\bf Proof of 1.} 
Suppose that $\nu$ is a probability measure satisfying Assumption \ref{assumption1} and
\[S^1 \subset \text{supp } \nu \subset 
\CC \setminus	\mathbb{D}.\] 
Recall that the random polynomials $P_n$ are
	\[ P_n(z) = \sum_{k=0}^{n}
	\frac{a_k}{\sqrt{\langle X^k, X^k \rangle_{n,\nu}}} z^k,  \]
	where
	\[
	\langle X^k, X^k \rangle_{n,\nu}
	=\int_{\mathbb C}
	|z|^{2k}
	e^{-2n V^\nu(z)} 
	\mathrm d \nu(z).\]
	We define
	\[ c_n= \int_\mathbb{C} e^{-2nV^{\nu}(z)}{\rm d}\nu(z), \]
	and we would like
	to prove that,
	almost surely, $(
	\sqrt {c_n}P_n)_{n \in \NN}$ converges uniformly 
	on compact sets of $\mathbb D$
	towards $z \mapsto \sum a_k z^k$. 
	Afterwards, we conclude
	by Lemma \ref{lemma:hurwitz}. Let $\nu_n$ be the probability measure
	defined by
	\[ {\rm d}\nu_n = \frac{1}{c_n} e^{-2nV^{\nu}} {\rm d}\nu.  \]
	We begin by proving that
	\[\nu_n
	\xrightarrow[n \to \infty]{
	\mathrm{weakly}} \nu_{S^1}   \] 
	where $\nu_{S^1}$ denotes
	the uniform measure on the unit circle.
	This will be a consequence of
	Laplace's method, some simple
	tightness property, and the invariance
	under rotations of the measures $\nu_n$. 
	Then, we expect
	that
	$\langle X^k,X^k\rangle_{n,\nu}
	/c_n
	= \int_\mathbb{C} |z|^{2k} 
	{\rm d}\nu_n(z)$
	converges to
	$\int_\mathbb{C} |z|^{2k} {\rm d}\nu_{S^1}(z)$. This may
	be a consequence of the weak convergence
	if $|z|^{2k}$ were bounded. By Laplace's method,
	we prove that
	the integral of $|z|^{2k}$ 
	outside a large open disk 
	converges to zero and then we may consider
	$|z|^{2k}$ as bounded.
	
	\textbf{Step 1: Convergence of the measures}
	
	\noindent In this step we prove that
	\[\nu_n
	\xrightarrow[n \to \infty]{
	\mathrm{weakly}} \nu_{S^1} .  \]
	First, let us show that 
	\[ \lim_{n\to \infty} 
	\frac{1}{n} \log c_n  = 0. \]
	This may be seen as an application 
	of Laplace's method but
	we write the proof for the reader's convenience.
	Since $V^{\nu}$ is non-negative, 
	$c_n \leq 1$, which implies
	\[ \varlimsup_{n\to \infty} 
	\frac{1}{n} \log c_n \leq 0.  \]
	Let $\varepsilon >0$ fixed. Then, since
	$V^{\nu}$ is continuous and equals $0$ on the unit circle, there exists $\delta >0$ such that $V^{\nu}(z) \leq \varepsilon$
	for all $z \in D_{1+\delta}$. 
	This implies that
	\[ c_n = \int_\mathbb{C} e^{-2n V^{\nu}(z)}{\rm d}\nu(z) \geq \int_{D_{1+\delta}} 
	e^{-2nV^{\nu}(z)}{\rm d}\nu(z) \geq e^{-2n\varepsilon} \nu(D_{1+\delta}). \]
	Taking the logarithm and the lower limit  we get,
	since $\nu(D_{1+\delta})>0$,
	\[ \varliminf_{n \to \infty} 
	\frac{1}{n} \log c_n \geq -2\varepsilon .\]
	Since this can be done for every $\varepsilon>0$,
	 we
	obtain 
	\[ \lim_{n \to \infty} \frac{1}{n} \log c_n  = 0. \]
	This behavior along with the fact that for any closed $A \subset \CC \setminus  \bar{\mathbb D}$ 
	\[ \frac{1}{n} \log \int_A e^{-2nV^{\nu}(z)} {\rm d}\nu(z)\leq - 2\inf_{z \in A} V^{\nu}(z) <0 \]
	imply that for any $r>1$
	\begin{equation}
	\label{eq: nun outside of br}
	\nu_n(\hspace{0.5pt} \CC \setminus D_r)
	 \xrightarrow[n \to \infty]{} 0. 
	\end{equation}
	This last fact also implies that the
	sequence is tight.
	Since every $\nu_n$ is invariant
	under rotations, every
	limit point of the sequence 
	$(\nu_n)_{n \in \NN}$ is 
	also invariant under rotations.
	The fact that
	$\nu_n(\mathbb D) = 0$
	and \eqref{eq: nun outside of br}
	imply that the only
	possible limit point is $ \nu_{S^1}$
	so that
	\[\nu_n
	\xrightarrow[n \to \infty]{
	\mathrm{weakly}} \nu_{S^1} .  \] 
	
	\textbf{Step 2: Convergence of the integrals}
	
	\noindent Now let us prove that for any fixed 
	non-negative integer $k$ we have
	\[ \frac{\langle X^k, X^k
	\rangle_{n,\nu}}{c_n} = 
	\int_\mathbb{C} |z|^{2k} {\rm d}\nu_n(z) 
	\xrightarrow[n \to \infty]{} 
	\int_\mathbb{C} |z|^{2k} {\rm d}\nu_{S^1}(z) = 1.  \]
	For $A>1$ we write
	\begin{align*}
	\int_\mathbb{C} |z|^{2k} {\rm d}\nu_n(z) 
	= \int_{D_A} |z|^{2k} {\rm d}\nu_n(z) 
	+ 
	\int_{\CC \setminus D_A} |z|^{2k} {\rm d}\nu_n(z).
	\end{align*}
	First, we notice that the convergence of $\nu_n$ towards $\nu_{S^1}$ implies that 
	\[ 
	\int_{D_A} |z|^{2k} {\rm d}\nu_n(z) \xrightarrow[n \to \infty]{}
	\int_{D_A} |z|^{2k} 
	{\rm d}\nu_{S^1}(z) = 1. \]
	That the integral on $\CC \setminus D_A$ is
	negligible follows from
	\[
	\lim_{n\to \infty}
	 \frac{1}{n}\log \int_{\CC \setminus D_A} |z|^{2k} e^{-2nV^{\nu}(z)}{\rm d}\nu(z)
	=
	-\inf_{z \in \CC \setminus D_A} V^\nu(z) \]	
	which is an application of Laplace's method.
	For convenience of the reader
	we will proceed in a somewhat more explicit way.
	Since $V^{\nu}(z) \sim_{|z| \to \infty} \log|z|$, 
	we may have chosen $A>1$ 
	such that for $|z|>A$ we have
	$V(z) \geq 1/2 \log |z|$.
	We obtain
	\begin{align*}
	\int_{\CC \setminus D_A} |z|^{2k} e^{-2nV^{\nu}(z)}{\rm d}\nu(z) & \leq  \int_{\CC \setminus D_A} |z|^{2k} e^{-n\log|z|}{\rm d}\nu(z) \\  
	&  = \int_{\CC \setminus D_A} |z|^{2k - n} {\rm d}\nu(z)\\
	&  \leq A^{2k-n}
	\end{align*}     
	if $n \geq 2k$.
	This entails that 
	\[ \varlimsup_n \frac{1}{n} \log  
	\int_{\CC \setminus D_A} |z|^{2k} e^{-2nV^{\nu}(z)}{\rm d}\nu(z) \leq -A  \]
	which, using the behavior of $c_n$,
	implies that 
	\[  \frac{1}{c_n}\int_{\CC \setminus D_A} |z|^{2k} e^{-2nV^{\nu}(z)}{\rm d}\nu(z) \xrightarrow[n \to \infty]{} 0.\]
	Hence, we obtained that for any fixed $k$,
	\[ \frac
	{\langle X^k,X^k\rangle_{n,\nu}}
	{c_n} 
	\xrightarrow[n \to \infty]{}
	1.  \]
	
	\textbf{Step 3: Uniform convergence
		of the polynomials}
	
	\noindent Let $\rho \in (0,1)$, then for any $z \in D_{\rho}$ we have
	\[\left| \sqrt{c_n} P_n(z) - \sum_{k=0}^{n} a_k z^k \right|\leq \sum_{k=0}^{n} \left|  
\sqrt{	
\frac{c_n}{\langle X^k ,
	X^k \rangle_{n,\nu}}}-1  \right| |a_k| \rho^k. \]
	This implies that, almost surely, 
	$(\sqrt{c_n}P_n)_{n \in \NN}$
	converges
	uniformly on
	$D_\rho$ towards \linebreak $z \mapsto \sum_{k=0}^{\infty} a_k z^k$ if we notice that 
	\[ \frac{c_n}
	{\langle X^k, X^k \rangle_n} 
	\leq 1 \] which allows us to use 
	Lebesgue's dominated convergence theorem. Since this happens
	for every $\rho \in (0,1)$ we have
	obtained that, almost surely,
	$(\sqrt{c_n}P_n)_{n \in \NN}$
	converges
	uniformly on
	compact sets of $\mathbb D$ towards  
	$z \mapsto \sum_{k=0}^{\infty} a_k z^k$.

	\textbf{Step 4: Convergence of the point process}
	
	\noindent To complete the proof, we use Hurwitz's Theorem, detailed in Lemma
	\ref{lemma:hurwitz}, which gives
	the almost sure convergence of the point process.
\end{proof}

\begin{proof}[Proof of the Corollary \ref{CorollairePoly}]
The proof is the same as the one of Corollary \ref{CorollaireJellium}. We use the inversion equivariance and the continuity of the minimum to obtain the convergence in law of 
$\min_{k \in \{1,\dots, n\}}{1/|z_k|}$ and we deduce the corollary by inverting again and scaling.
\end{proof}

\subsubsection{Proof of Theorem \ref{BulkMinMaxPoly}}

\begin{proof}
	\textbf{Proof at the origin.}
	Let us define 
	the positive measure $\mu_{\alpha}$ by
		\[ {\rm d}\mu_\alpha (z)
	= \frac{\alpha}{2\pi} 
	|z|^{\alpha-2} {\rm d}\ell_{\mathbb{C}}(z)\]
	which is the only measure $\mu$ satisfying
	that for every $r>0$, $\mu(D_r) = r^{\alpha}$.
	A change of variables
	 shows
	that
\begin{equation}\label{Mittag}
\int_\mathbb{C} |z|^{2k} e^{-\frac{\lambda}{\alpha}|z|^{\alpha}}\lambda {\rm d}\mu_{\alpha}(z)
= 
\alpha
\left( \frac{\alpha}{\lambda}\right)^{2k/\alpha} 
\Gamma \left(1+\frac{2k}{\alpha}\right).
\end{equation}
	Using Lemma \ref{lemma:hurwitz}
	and Lemma \ref{lemma:min}, it is
	enough to prove that, almost surely,
	\[\frac{1}{\sqrt{n}}P_n \left(\frac{z}{n^{1/\alpha}}
	\right) = \sum_{k=0}^n a_k \frac{1}{\sqrt{n}}\frac{1}{n^{k/\alpha}
	\sqrt{\langle X^k , X^k \rangle_{n,\nu}}}z^k \]
	converges uniformly on compact sets of $\mathbb{C}$ towards	
	\[ f_{\alpha,\lambda}(z)= \sum_{k=0}^{\infty} \frac{a_k}{\displaystyle \left( \int_\mathbb{C} |z|^{2k} e^{-\frac{\lambda}{\alpha}|z|^{\alpha}}\lambda {\rm d}\mu_{\alpha}(z) \right)^{1/2} }z^k .\]
	We start by recalling some properties of the potential $V^{\nu}$. Indeed,
	for the convergence to hold, 
	we assume
	$V^{\nu}(0)=0$ which can be
	done by adding a constant. 
	Then we will prove that, for any $k$,
	\begin{equation}
	\label{eq:coeffPoly}
	n^{1+2k/\alpha}\langle X^k,X^k \rangle_{n,\nu} =
	 n^{2k/\alpha}\int_\mathbb{C} |z|^{2k} e^{-2nV^{\nu}(z)}n\, {\rm d}\nu(z) \xrightarrow[n \to \infty]{} \int_\mathbb{C} |z|^{2k} e^{-\frac{\lambda}{\alpha}|z|^{\alpha}}\lambda {\rm d}\mu_{\alpha}(z). 
	\end{equation}
	The idea is quite simple.
	If $T_n(z) = n^{1/\alpha} z$ then
	\[n^{2k/\alpha}\int_\mathbb{C} |z|^{2k} e^{-2nV^{\nu}(z)}n\, {\rm d}\nu(z) = 
	\int_\mathbb{C} |z|^{2k} 
	e^{-2nV^{\nu}\left( \frac{z}{n^{1/\alpha}} \right)}
	n\, \mathrm d T_n(\nu)(z)\]
	where $T_n(\nu)$ denotes the pushforward
	 measure
	of $\nu$ by $T_n$.
	By the hypothesis 
	\eqref{eq:HypZero},	
	we should have that
	$nT_n(\nu)$ converges towards $\lambda \mu_\alpha$
	and $nV^{\nu}\left( \frac{z}{n^{1/\alpha}} \right)$
	converges towards $(\lambda/\alpha)|z|^\alpha$
	in some sense what would imply \eqref{eq:coeffPoly}. 
	Finally, we find a sequence $B_k$ such that for any $n \in \NN$,
	\[ \frac{1}{n^{\frac{1}{2}+
	\frac{k}{\alpha}}
	\sqrt{\langle X^k, X^k\rangle_{n,\nu}}} \leq B_k   \]
	with $\sum_{k=0}^{\infty} a_k B_k z^k$ having,
	almost surely, 
	an infinite radius of convergence
	which,
	by
	Corollary
	\ref{cor:ConvergenceRadius},
	 happens if and only if
	$\sum_{k=0}^{\infty} B_k z^k$
	has an infinite radius
	of convergence.
	
	\textbf{Step 1: Properties of the potential} 
	
	\noindent Let $\nu$ be a rotationally invariant probability measure and suppose
that there exists $\alpha >0$ and 
	$\lambda >0$ such that
	\[ \lim_{r \to 0} \frac{\nu(D_r)}{r^{\alpha}} = \lambda.  \]
	We will assume that $V^{\nu}(0)=0$, since adding a constant to the potential $V^{\nu}$ only changes the polynomials $R_{k,n}$ 
	by a multiplicative constant
	(depending only on $n$) 
	which has no impact on the zeros of $P_n$.
	So, using 
	Lemma \ref{Formula potential}, we can write
	\[ V^{\nu}(z) = \int_{0}^{|z|} \frac{\nu(D_r)}{r} {\rm d}r. \]
	We obtain that 
	\[ \frac{ V^{\nu}(r)}{r^{\alpha}}
	\xrightarrow[r \to 0]{} \frac{\lambda}{\alpha}
	=:\gamma  \]
	and that $V^{\nu}(r) > 0$ for every
	$r > 0$.
	We may also obtain a useful lower bound
	for $V^\nu$. If $\delta \in (0,1)$ 
	and if $|z|\geq \delta$ we can write
	\begin{align}
	\label{eq:VgeqNuLog}
	V^{\nu}(z) & = \int_0^{\delta} \frac{\nu(D_r)}{r}
	{\rm d}r +\int_\delta^{|z|} 
	 \frac{\nu(D_r)}{r}{\rm d}r  \nonumber \\
	& \geq \nu(D_{\delta}) \log |z| - \nu(D_{\delta}) \log \delta \nonumber \\
	& \geq  \nu(D_{\delta}) \log |z|
	\end{align}
	where we have used that $\log \delta < 0$.

	\textbf{Step 2: Convergence of the coefficients}
	
	\noindent Let us define 
	$T_n: z \mapsto n^{1/\alpha} z$. Then for any $r>0$ we have
	\[ n T_n(\nu)(D_r) = n \nu(D_{r/n^{1/\alpha}}) \xrightarrow[n \to \infty]{} \lambda r^{\alpha}=
	\lambda \mu_{\alpha}(D_r) .  \]
	In particular, $nT_n(\nu)(D_K)$ converges towards $\lambda \mu_{\alpha}(D_K)$ and the cumulative distribution function of 
	$nT_n(\nu)/nT_n(\nu)(D_K)$, which
	is a probability measure on $D_K$, converges pointwise towards the cumulative distribution function of $\lambda \mu_{\alpha}/ \lambda \mu_{\alpha}(D_K)$.
	This implies that, for any $K > 0$ and any
	bounded continuous function $g$
	on $D_K$ we have
	\[ \int_{D_K} g \, n {\rm d}T_n(\nu) \xrightarrow[n \to \infty]{} \lambda \int_{D_K} g {\rm d}\mu_{\alpha} .
	\]
	Let $\varepsilon \in (0,1)$. There exists 
	$\delta \in (0,1)$ 
	such that for any $z \in D_{\delta}$ we have
	\[ (1-\varepsilon) \gamma |z|^{\alpha } \leq V(z)  \leq (1+\varepsilon) \gamma |z|^{\alpha } \]
	and for any $r \in (0,2\delta)$ we have
	\begin{equation}
	\label{eq:nuDlambda}
	\nu(D_r) \leq 2 \lambda r^{\alpha} . 
	\end{equation}	  
	Let us decompose
	the integral 
	that interests us,
	the left-hand side of \eqref{eq:coeffPoly},
	as the sum 
	\begin{align*}
	\int_\mathbb{C} 
	\! n^{2k/\alpha} |z|^{2k} e^{-2nV^{\nu}(z)}n{\rm d}\nu(z)  = & \int_{D_{\delta}}\! n^{2k/\alpha} |z|^{2k} e^{-2nV^{\nu}(z)}n{\rm d}\nu(z) \\ 
	& + \int_{\CC \setminus D_{\delta}}\! n^{2k/\alpha} |z|^{2k} e^{-2nV^{\nu}(z)}n{\rm d}\nu(z).
	\end{align*}
	From the lower bound 	\eqref{eq:VgeqNuLog}
	found 
	in Step 1, we know that 
	if we denote $\kappa = \nu(D_\delta)$ then
	$V^{\nu}(z) \geq  \kappa\log|z|$ 
	for $|z| \geq \delta$. 
	This implies that 
	\[ 0 \leq \int_{\CC \setminus D_{\delta}}\! n^{2k/\alpha} |z|^{2k} e^{-2nV^{\nu}(z)}n{\rm d}\nu(z) \leq   n^{2k/\alpha+1} \delta^{2k-2\kappa n}   \]
	when $n \geq k/\kappa$, since in that case
	$|z|^{2k-2\kappa n} \leq \delta^{2k-2\kappa n} $
	for $|z| \leq \delta$, and then		
	\[\int_{\CC \setminus D_{\delta}}\! n^{2k/\alpha} |z|^{2k} e^{-2nV^{\nu}(z)}n{\rm d}\nu(z)
	\xrightarrow[n \to \infty]{} 0.\]
	To study the first term, 
	the integral over $D_\delta$,
	we start by noticing that 
	\begin{equation}
	\label{eq:IntegralInequality1}
	 \int_{D_{\delta}}\! n^{2k/\alpha} |z|^{2k} e^{-2n\gamma(1+\varepsilon)|z|^\alpha}n{\rm d}\nu(z)\leq  \int_{D_{\delta}}\! n^{2k/\alpha} |z|^{2k} e^{-2nV^{\nu}(z)}n{\rm d}\nu(z) 
	 \end{equation}
	and 
	\begin{equation}
	\label{eq:IntegralInequality2}	
	  \int_{D_{\delta}}\! n^{2k/\alpha} |z|^{2k} e^{-2nV^{\nu}(z)}n{\rm d}\nu(z) \leq  \int_{D_{\delta}}\! n^{2k/\alpha} |z|^{2k} e^{-2n\gamma(1-\varepsilon)|z|^\alpha}n{\rm d}\nu(z). 
	 \end{equation}  
	If we prove that for any $\gamma' >0$, we have
	\begin{equation}
	\label{eq:IntegralOverD}
	\int_{D_{\delta}}\! n^{2k/\alpha} |z|^{2k} e^{-2n\gamma'|z|^\alpha}n{\rm d}\nu(z) \xrightarrow[n \to \infty]{} \int_{\mathbb{C}}\! |z|^{2k} e^{-\gamma'|z|^\alpha}\lambda {\rm d}\mu_{\alpha}(z)
	\end{equation} 
	then,
	using also
	\eqref{eq:IntegralInequality1}
	and 
	\eqref{eq:IntegralInequality2}, 
	we will obtain that for any $\varepsilon >0$
	\[ \varlimsup_{n} 
	\int_{D_{\delta}}\! n^{2k/\alpha} |z|^{2k} 
	e^{-2nV^\nu(z)}n{\rm d}\nu(z) \leq  \int_{\mathbb{C}}\! |z|^{2k} e^{-\gamma(1-\varepsilon)|z|^\alpha}\lambda {\rm d}\mu_{\alpha}(z)\]
	and
	\[\varliminf_{n} \int_{D_{\delta}}\! n^{2k/\alpha} |z|^{2k} e^{-2nV^\nu(z)}n{\rm d}\nu(z) \geq  \int_{\mathbb{C}}\! |z|^{2k} e^{-\gamma(1+\varepsilon)|z|^\alpha}\lambda {\rm d}\mu_{\alpha}(z).\]
	Taking the limit as $\varepsilon$ goes to zero
	will complete the proof.
	
	To prove \eqref{eq:IntegralOverD} let us write 
	\begin{align*} \int_{D_{\delta}}\! n^{2k/\alpha} |z|^{2k} e^{-2n\gamma'|z|^\alpha}n{\rm d}\nu(z) & = \int_{D_{\delta n^{1/\alpha}}}\! |x|^{2k} e^{-2n\gamma'|x/n^{1/\alpha}|^\alpha}n
	{\rm d}T_n(\nu)(x) \\
	&  = \int_{D_{\delta n^{1/\alpha}}}\! |x|^{2k} e^{-2\gamma'|x|^\alpha}n
	{\rm d}T_n(\nu)(x).
	\end{align*}
	For any fixed integer $K>0$,
	the weak convergence of 
	the measures $(nT_n(\nu))_{n \in \NN^+}$ towards 
	$\lambda \mu_{\alpha}$ implies that 
	\begin{equation}\label{goal}
	\int_{D_{K}}\! |x|^{2k} e^{-2\gamma'|x|^\alpha}n{\rm d}T_n(\nu)(x) \xrightarrow[n \to \infty]{}  \int_{D_{K}}\! |x|^{2k} e^{-2\gamma'|x|^\alpha}\lambda {\rm d}\mu_{\alpha}(\nu)(x).
	\end{equation}
	If we are able to find a function $h(K)$ going to zero as $K$ goes to infinity for which
	\[ 0 \leq \int_{D_{\delta n^{1/\alpha}}\setminus D_K}\! |x|^{2k} e^{-2\gamma'|x|^\alpha}n
	{\rm d}T_n(\nu)(x) \leq h(K)  \]
	for $n$ large enough then \eqref{eq:IntegralOverD} would be established. To this aim, we write
	\begin{align*}
	\int_{D_{\delta n^{1/\alpha}}\setminus D_K}\! \! \! |x|^{2k} e^{-2\gamma'|x|^\alpha}n
	{\rm d}T_n(\nu)(x) & \leq \! \sum_{j=K}^{ 
		\lfloor \delta n^{1/\alpha} \rfloor} 
	\int_{D_{j+1} \setminus D_j} \! \! |x|^{2k} e^{-2\gamma'|x|^\alpha}n {\rm d}T_n(\nu)(x) \\ 
	& \leq \sum_{j=K}^{ 
		\lfloor \delta n^{1/\alpha} \rfloor} 
	(j+1)^{2k} e^{-2\gamma'j^\alpha} nT_n(\nu)
	(D_{j+1}) \\
	& \leq \sum_{j=K}^{
		\lfloor \delta n^{1/\alpha} \rfloor} 
	(j+1)^{2k} e^{-2\gamma'j^\alpha} 
	n\nu(D_{(j+1)/{n^{1/\alpha}}}) \\
	& \leq  \sum_{j=K}^{ 
		\lfloor \delta n^{1/\alpha} \rfloor} 
	(j+1)^{2k} e^{-2\gamma'j^\alpha} \lambda 2
	(j+1)^{\alpha} \\
	& \leq \sum_{j=K}^{\infty} (j+1)^{2k} e^{-2\gamma'j^\alpha} \lambda 2(j+1)^{\alpha} = h(K)
	\end{align*}
	where we have used that
	$\lfloor \delta n^{1/\alpha} \rfloor + 1
	\leq 2\delta n^{1/\alpha}$
	for $n$ large enough to apply
	\eqref{eq:nuDlambda}.
	This ends the proof of this step.
	
	\textbf{Step 3: Dominated convergence}
	
	\noindent To obtain the uniform convergence of 
	$P_n(./n^{1/\alpha})/\sqrt n $ towards $f_{\alpha,\lambda}$, it suffices to find a sequence 
	$(B_k)_{k \in \NN}$, 
	independent of $n$, such that for any $k$ and $n$
	with $k \leq n$,
	\begin{equation}
	\label{eq:Bk2} {B_k}^{-2} \leq
	n^{1+\frac{k}{\alpha}}
	\langle X^k,X^k\rangle_{n,\nu}
	\end{equation}
	and such that the power series $\sum_{k=0}^{\infty} 
	B_k z^k$ have an infinite radius of convergence.
	
	Let $\varepsilon >0$ such that for any $r \in [0,\varepsilon]$
	\begin{equation}\label{inequality measures}
	\frac{3\lambda}{4} r^\alpha \leq \nu(D_r) \leq \frac{5\lambda}{4} r^\alpha
	\end{equation}
	and for any $z \in D_{\varepsilon}$
	\[  \frac{\gamma}{2} |z|^\alpha \leq V^{\nu}(z) \leq \frac{3\gamma}{2} |z|^{\alpha}.  \]   
	Such an $\varepsilon$ exists 
	since
	$\lim_{r\to 0} \nu(D_r)/\mu_\alpha(D_r)= 
	\lambda $ and
	$ \lim_{r \to 0} r^{-\alpha} V^{\nu}(r)
	=\gamma  $.
	If $k \leq n$, \linebreak we have
	\begin{align*}
	\int_\mathbb{C} n^{2k/\alpha} |z|^{2k} e^{-2nV^{\nu}(z)} n {\rm d}\nu(z) & = \int_{\CC} 
	|x|^{2k} e^{-2n V^\nu
	\left( \frac{x}{n^{1/\alpha}}
	 \right)} 
	n \, {\rm d}  T_n(\nu)(x) \\
	& \geq \int_{D_{\varepsilon n^{1/\alpha}} } |x|^{2k} e^{-3\gamma |x|^\alpha} n\,{\rm d}T_n(\nu)(x)  \\
	& \geq 
	\int_{D_{\varepsilon k^{1/\alpha}} \setminus  D_{\varepsilon (k/2)^{1/\alpha}}} |x|^{2k} e^{-3\gamma |x|^\alpha} n\, {\rm d}T_n(\nu)(x) \\
	\geq \left( \frac{\varepsilon}{2} \left( \frac{k}{2} \right)^{1/\alpha} \right)^{2k} &e^{-3\gamma \varepsilon^{\alpha} k } \left( n T_n(\nu) 
	\left(D_{\varepsilon k^{1/\alpha}
	}\right) -  n T_n(\nu)
	\left(
	D_{\varepsilon 
	(k/2)^{1/\alpha}}\right) \right).
	\end{align*}
	Due to the inequality \eqref{inequality measures}, we deduce that, for $k \geq n$,
	\begin{align*}
	n T_n(\nu) \left(D_{\varepsilon k^{1/\alpha}}\right) -  
	n T_n(\nu)(D_{\varepsilon (k/2)^{1/\alpha}}) & = 
	n\nu\left(D_{\varepsilon k^{1/\alpha} n^{-1/\alpha}}\right) - n\nu \left(D_{\varepsilon (k/2)^{1/\alpha}
	 n^{-1/\alpha}} \right) \\
	& \geq \frac{3 \lambda}{4}  \varepsilon^\alpha k - \frac{5
	\lambda }{4} 
\varepsilon^\alpha 
	\frac{k}{2} \\
	& \geq k \frac{
	\lambda  \varepsilon^\alpha}{8}.
	\end{align*}
If we define $B_k$ such that
	\[ B_k^{-2} =\left( \frac{\varepsilon}{2} 
	\left(\frac{k}{2}\right)
	^{1/\alpha} \right)^{2k} e^{-3\gamma \varepsilon^{\alpha} k }  k \frac{\lambda  \varepsilon^\alpha}
	{8}, \]
then we have
\[\int_\mathbb{C} n^{2k/\alpha} |z|^{2k} e^{-2nV^{\nu}(z)} n {\rm d}\nu(z) 
	\geq B_k^{-2}.\]
	which is
	\eqref{eq:Bk2}.
Since
	\[ \frac{1}{k} \log B_k \xrightarrow[k \to \infty]{} -\infty
	 , \]
	we have that $\sum_{k=0}^{\infty} 
	B_k z^k$ has an infinite radius of convergence.
	Let $\rho > 0$.
	For every $z \in D_\rho$ we have
\begin{align*} \sum_{k=0}^{\infty} &
\left(\frac{1}{\displaystyle
n^{\frac{1}{2}+\frac{k}{\alpha}}
\sqrt {\langle X^k,X^k
 \rangle_{n,\nu}}}
- \frac{1}{\displaystyle \left( \int_\mathbb{C} |z|^{2k} e^{-\frac{\lambda}{\alpha}|z|^{\alpha}}			\lambda {\rm d}\mu_{\alpha}(z) \right)^{1/2} } 
\right) a_k z^k   				\\
& \quad  \quad  \quad \quad \leq  
\sum_{k=0}^{\infty}
\left|\frac{1}{\displaystyle
n^{\frac{1}{2}+\frac{k}{\alpha}}
\sqrt {\langle X^k,X^k
 \rangle_{n,\nu}}}
- \frac{1}{\displaystyle \left( \int_\mathbb{C} |z|^{2k} e^{-\frac{\lambda}{\alpha}|z|^{\alpha}}			\lambda {\rm d}\mu_{\alpha}(z) \right)^{1/2} } 
\right||a_k| \rho^k   .
\end{align*}
The $k$-th term 
of the right-hand side
is dominated by
$2B_k |a_k| \rho^k$ so that, by
using Lebesgue's dominated convergence theorem, we obtain that,
almost surely,  
	\[ \frac{1}{\sqrt n} P_n \left(\frac{z}{n^{1/\alpha}}
	\right) \xrightarrow[n \to \infty]{} \sum_{k=0}^{\infty} \frac{a_k}{\displaystyle \left( \int_\mathbb{C} |z|^{2k} e^{-\frac{\lambda}{\alpha}|z|^{\alpha}}			\lambda {\rm d}\mu_{\alpha}(z) \right)^{1/2} }z^k  
	 \]
	uniformly on $D_\rho$.
	Since this happens
	for every $\rho>0$
	and by writing
	the explicit expression
	of the integral
	\eqref{Mittag}
	we obtain that, almost surely,
\[ \frac{1}{\sqrt n} P_n \left(\frac{z}{n^{1/\alpha}}
	\right) 
	\xrightarrow[n \to \infty]{}
	\sum_{k=0}^{\infty} \frac{a_k}{\alpha^{1/2}
\left( \frac{\alpha}{\lambda}\right)^{k/\alpha} 
\Gamma \left(1+\frac{2k}{\alpha}\right)^{1/2}} 
 z^k
=\frac{1}{\alpha^{1/2}}
 f_{\alpha,\lambda}(z)\]	
uniformly on compact sets of 
$\CC$.

	\textbf{Proof at infinity and of
	the maxima. } 
	By the equivariance under inversion, both
	points are an immediate consequence of the first point, using the fact that if $\nu$ satisfies $\lim_{r \to \infty} r^{\alpha} \nu(\hspace{0.5pt} \CC \setminus D_r) = \lambda$, then its pushforward by the inversion, $i_*\nu$, satisfies $\lim_{r \to 0} i_*\nu(D_r)/r^{\alpha} = \lambda$.
	
\end{proof}

\subsubsection{Proof of Theorem \ref{th:Weyl}}

As promised, this will be a very short proof
which will use the same ideas of the previous
proofs. Notice that $P_n$ has the same law as
  \[\tilde P_n(z) = 
\sum_{k=0}^n 
\frac{\sqrt{n^k}}{\sqrt {k!}} a_{n-k} z^k.\]
We invert the zeros by considering
\[Q_n(z) = \frac{\sqrt{n!}}{\sqrt{n^n}} 
z^n \tilde P_n(1/z)
=\sum_{k=0}^n 
\frac{\sqrt {n(n-1)\dots (n-k+1)}}{
\sqrt{n^{k}}} a_{k} z^k.\]
Notice that
\[\frac{\sqrt {n(n-1)\dots (n-k+1)}}{
\sqrt{n^{k}}}\leq 1 \ \ \mbox{ and } \ \
\lim_{n \to \infty}
\frac{\sqrt {n(n-1)\dots (n-k+1)}}{
\sqrt{n^{k}}} = 1\]
for any $k$. By a dominated convergence argument
and since the series
\begin{equation}
\label{eq:series}
\sum_{k=0}^\infty a_k z^k
\end{equation}
has a radius of convergence one, we obtain that,
almost surely,
$Q_n$ converges uniformly on compact sets
of $\mathbb D$
towards the random analytic function defined by
\eqref{eq:series}. The proof is completed
by the Hurwitz's continuity
(Lemma \ref{lemma:hurwitz}), the continuity
of the minimum (Lemma \nolinebreak \ref{lemma:min})
and the determinantal
structure of the zeros if
the coefficients are complex Gaussian random variables.

\section{Appendix: Point processes}
\label{section:appendix}
We remind
some definitions and properties
of point processes
that are used in the article.
Let $X$ be a Polish space\footnote{We
say that a separable topological space $X$ 
is a Polish space if there exists
a complete metric that metrizes its topology. It is not necessary to choose one such metric but
only to know that it exists.}.
We denote by $\mathcal C_X$
the space of locally
finite positive measures $P$ on $X$
such that $P(A)$ is a non-negative
integer or infinity
for every
measurable set $A \subset X$.
It is not hard to see
that
for every $P$ there exists
a countable family 
$(x_\lambda)_{\lambda \in \Lambda}$
of elements of $X$ 
such that every $x \in X$
has an open neighborhood 
$U \subset X$ for which
the cardinal of
$\{\lambda \in \Lambda :
x_\lambda \in U\}$
is finite
and such that
\[P = \sum_{\lambda \in \Lambda}
\delta_{x_\lambda}.\]
Indeed, we could have 
defined	$\mathcal C_X$
more loosely by saying
\[ \mathcal C_X = 
\left\{P \subset X :
P \mbox{ is locally finite
	and admits multiplicities} \right\} \]
and the measure version $P$ 
would count the number of points
inside a set.
This set notation shall be used
along the article.
We will endow $\mathcal C_X$
with a topology.
Let $f:X \to \mathbb R$
be a continuous function with compact support.
Define
$\hat f:\mathcal C_X \to \mathbb R$
by
\[ \hat f(P)=\sum_{x \in P} f(x)
= \int_X f {\rm d}P \]
where in the sum we count $x$
with multiplicity.
Notice that $\hat f$ makes sense
since
$P$ is locally finite and
$f$ is compactly supported. Then
we endow $\mathcal C_X$
with the smallest
topology such that 
$\hat f$ is continuous 
for every continuous function 
$f :X \to \mathbb R$
with compact support.
Notice that
this topology
is the vague topology
if 
$\mathcal C_X$ is seen as
a subspace of
the space of Radon measures on $X$.
In particular,
since the space of Radon measures
on a Polish space is Polish 
\cite[Theorem 4.2]{Kallenberg},
it can be proved that
$\mathcal C_X$,
since it is a closed subset of this space, is a
Polish space too. Finally,
a random element of $\mathcal C_X$
will be said to be
a \textit{point process} on $X$.

We shall be mainly interested
in the cases where $X$ is
$\mathbb R^+=
[0,\infty)$,
$X$ is $[0,1)$
or $X$ is an open subset of $\mathbb C$.
Below we state some facts that are
 used in the article.
 The first one is that convergence
on $\mathcal C_X$ implies
convergence of the minima.
The second one is that the application
that associates
to each holomorphic function its zeros is continuous.
The last one is the notion of determinantal
point process and some of its properties.

\subsection{Convergence of the minima}

\begin{lemma}[Continuity of the minimum]
	\label{lemma:min}
	The application
	$\min:\mathcal C_{\mathbb R^+} 
	\to [0,\infty]$,
	that to each $P \in \mathcal C_{\mathbb R^+}$
	associates its minimum or,
	in measure terms, the infimum
	of its support,
	is continuous.
	Similarly,
	the application
	that to each
	$P \in \mathcal C_{\mathbb [0,1) }$
	associates its minimum in $[0,1]$
	is continuous.
	\begin{proof}
		We will only give the proof
		of the continuity of $\min$
		since the proof of the continuity of
		the second
		application follows the same steps.
		Let $P \in  \mathcal C_{\mathbb R^+}$
		and consider
		a sequence $\{P_n\}_{n \in \mathbb N}$
		that converges to $P$.
		
		{\bf Suppose $\min P < \infty$}.
		Take $\varepsilon > 0$ and
		any  positive continuous function 
		$f: \mathbb R^+ \to \mathbb R$
		supported on $[0,\min P + \varepsilon]$
		such that $f(\min P) > 0$.
		Since $\lim_{n\to \infty}\hat f (P_n)=
		 \hat f (P)$,
		we have
		that
		$\hat f(P_n) >0$ for $n$ large enough
		and then $\min P_n \leq \min P + \varepsilon$
		for $n$ large enough so that
		\[\varlimsup_{n \to \infty} \min P_n
		\leq \min P + \varepsilon.\]
		Since 
		this can be done for every $\varepsilon>0$,
		we obtain \[\varlimsup_{n \to \infty} \min P_n
		\leq \min P.\]
		
		Now take $a \in \mathbb R^+$
		such that
		$a  < \min P$. 
		Consider any continuous function $f$
		supported on $[0,\min P]$
		such that $f(r) = 1$ if
		$r \leq a$. Since
		$\lim_{n \to \infty}
		\hat f (P_n)=
		\hat f (P) = 0$,
		we have that 
		$\hat f (P_n) < 1$ for $n$ large enough.
		Then, $P_n ([0,a]) = 0$
		and, thus, $\min P_n \geq a$
		for $n$ large enough. So,
		\[  a \leq \varliminf_{n \to \infty} 
		\min P_n.\]
		Since this can be done for 
		every $a < \min P_n$, we obtain
		\[\min P \leq \varliminf_{n \to \infty} 
		\min P_n\] and we may conclude.
		
		{\bf Suppose $ \min P = \infty$}, i.e.
		$P(\RR^+) = 0$.
		Take $M>0$ and consider
		a non-negative continuous function
		$f : \RR^+ \to \RR$ 
		with compact support such that
		$f(y) = 1$ if $y \leq M$.
		Then,
		since $\lim_{n \to \infty}
		\hat f(P_n) = \hat f(P)=0$,
		we have that
		$\hat f_n(P)< 1$ for $n$ large enough.
		In particular $P_n([0,M]) = 0$ 
		for $n$ large enough
		which implies that
		$\min P_n > M$ for those $n$.
		Since this can be done
		for every $M > 0$,
		we obtain \[\lim_{n \to \infty} \min P_n
		= \infty\] by definition of limit.
		
	\end{proof}
\end{lemma}

\subsection{Continuity of the zeros}

\begin{lemma}[Hurwitz's continuity]
	\label{lemma:hurwitz}
	Consider an open subset $U$ 
	of $\CC$
	and denote by $\mathcal O(U)$
	the space
	of not identically zero holomorphic functions
	endowed
	with the compact-open topology 
	(the topology of uniform
	convergence on compact sets).
	Then the map \break
	$\mathcal Z:\mathcal O(U)
	\to \mathcal C_{U}$
	defined
	by
	\[\mathcal Z(p) = 
	\sum_{p(z) = 0} \delta_z,\]
	where the zeros
	are counted with multiplicity,
	is continuous.
	\begin{proof}
Notice that, by the very
definition of the topology on
$\mathcal C_U$, the map
$\mathcal Z$ is continuous
if and only if
$\hat f \circ \mathcal Z$ is continuous
for every continuous
function $f:U \to \mathbb R$
with compact support.

	Let $f:U \to \RR$ be a 
		continuous
		function with compact support
		and
		let $(p_n)_{n \in \NN}$ 
		be a sequence of elements in 
		$\mathcal O(U)$
		that has a limit $p \in \mathcal O(U)$.
		Denote by $z_1, \dots ,z_l$
		the zeros of $P$ inside $\mbox{supp}(f)$.
		Denote by $L$ the total number of
		zeros of $P$ counted with multiplicity.
		Take $\varepsilon>0$. 
		We will find $N>0$ such that
		$\left|\sum_{p_n(z)=0}f(z) - 
		\sum_{p(z)=0}f(z)\right| 
		< \varepsilon$ for $n > N$
		where the zeros are, again,
		counted with multiplicity in the sums.
		
		By the continuity of $f$
		we can choose
		$\tilde \delta> 0$
		such that 
		for
		every $k \in \{1, \dots ,l\}$ we have
		$|f(z)- f(z_k)| < \varepsilon/L$
		for every $z \in D_{\tilde \delta}(z_k)$.
		By Hurwitz's theorem,
		since $(p_n)_{n \in \NN}$ converges
		to $p$ uniformly on compact sets,
		there exists $\delta>0$
		and $\tilde N>0$
		such that $\delta<\tilde \delta$
		and
		for
		every $k \in \{1, \dots ,l\}$
		the number
		of zeros of $p_n$ inside $D_\delta(z_k)$
		counted with multiplicity
		is exactly the same as the multiplicity of
		the zero $z_k$ of $p$ for $n \geq \tilde N$.
		Define
		$K=\mbox {supp} f 
		\cap \mathbb C \setminus D_{\delta}(z_1)
		\cap \dots
		\cap \mathbb C \setminus D_{\delta}(z_l)$.
		Because of the uniform convergence
		on $K$
		and because
		$p(z)\neq 0$
		for every $z \in K$ 
		we can take $N>0$ such that $N > \tilde N$
		and
		$\sup_{z \in K}|p_n(z) - p(z)| < 
		\inf_{z \in K}|p(z)|$. This implies,
		in particular, that
		$p_n(z) \neq 0$ for every $n \geq N$
		and $z \in K$.
		We may conclude by
		saying that
		
		\begin{align*}
		\left|\sum_{p_n(z)=0}f(z) - 
		\sum_{p(z)=0}f(z)\right| 
		&\leq 
		\sum_{k=1}^l 
		\left( \sum
		\limits_{\substack{p_n(z)=0 \\ 
				z \in D_{\delta}(z_k)}}
		\left|f(z) - f(z_k)\right| \right) \\
		& < 
		\sum_{k=1}^l 
		\left( \sum
		\limits_{\substack{p_n(z)=0 \\ 
				z \in D_{\delta}(z_k)}}
		\varepsilon/L \right)
		= \varepsilon.
		\end{align*}

	\end{proof}
	
\end{lemma}

\subsection{A small detour to determinantal
point process}

\label{sub:DPP}

Suppose $\mathcal X$ is a point process
on an open set $U \subset \mathbb C$.
Given a function $K: U \times U \to \mathbb C$,
we will say that $\mathcal X$ is a 
\textit{determinantal
point process} associated to the kernel $K$
(with respect to the Lebesgue measure) if the
following is true.
For every $m \in \NN^+$ and 
every disjoint measurable subsets
$A_1,\dots,A_m \subset U$ we have
\begin{equation}
\label{eq:ExpectedNumber}
\mathbb E[\# A_1 \dots \# A_m]
= 
\int_{\CC^m} 
\det (K(x_i,x_j))_{1 \leq i,j \leq m}
\mathrm d \ell_{\CC^m}(x_1,\dots,x_m)
\end{equation}
where $\# A_i$ denotes the number
of points of $\mathcal X$ inside $A_i$.
As a particular example, we may consider 
our Coulomb gases, i.e. 
$(x_1,\dots,x_n)$ following the law \eqref{gaz}.
Then the point process
$\mathcal X = \{x_1,\dots,x_n\}$
is a determinantal point process associated to
the orthogonal projection $K_n$ 
onto the space of holomorphic functions on $\mathbb C$
with weight $e^{-2(n+1)V(z)}$ or, more explicitly,
\[K_n(z,w)= \sum_{k=0}^{n-1}
b_{k,n} z^k \bar w^k e^{-(n+1)V(z)}
e^{-(n+1)V(w)}
\]
and
\[b_{k,n}
= \left( \int_\CC |z|^{2k} e^{-2(n+1)V(z)}
\right)^{-1}.\] 
For more details, we can see \cite{HoughKrisPeresVirag}. As other examples
we have the Bergman point processes
given on Definition \ref{def:BergmanPointProcess}.
From the very definition we can also
see that if $\mathcal X$ is a determinantal
point process on $U$ associated to the kernel $K$
and if $U' \subset U$ is an open subset then
$\mathcal X \cap U'$ is a determinantal
point process on $U'$ associated to the kernel
$K|_{U' \times U'}$.

There are two main properties that we
will need in our proofs.
The first one is that
the class of determinantal point process
with respect to the Lebesgue measure
is invariant under diffeomorphisms.
We will only need the following stronger
invariance under biholomorphisms. 

\begin{lemma}[Change of variables formula]
\label{lem:ChangeOfVariables}
Let
 $\mathcal X$ be a determinantal point process
on $U$
associated to the kernel $K$
and
suppose that $f:U \to \tilde U$
is a biholomorphism. Then
$f(\mathcal X)$ is a determinantal
point process on $\tilde U$
associated to the kernel
$\tilde K: \tilde U \times \tilde U \to \mathbb R$
given by
\[\tilde K(x,y) = 
(f^{-1})'(x)
K(f^{-1}(x),f^{-1}(y))
\overline {(f^{-1})'}(y) \]
\begin{proof}
It is a straightforward calculation
using \eqref{eq:ExpectedNumber}.
\end{proof}
\end{lemma}

The second one
is contained
in \cite{ShiraiTakahashi} and is
the main tool to prove the convergence of 
our point processes.

\begin{proposition}[Convergence of determinantal
point processes]
\label{prop:ConvergenceOfDPP}
Suppose that 
$(\mathcal X_n)_{n \in \NN}$ 
is a sequence
of determinantal point processes
on an open subset $U \subset \CC$
associated
to a sequence
of continuous kernels 
$(K_n)_{n \in \NN}$.
If there exists $K: U \times U \to \mathbb C$
such that
\[\lim_{n\to\infty} K_n = K\]
uniformly on compact sets of 
$U \times U$ then
there exists a determinantal point process
$\mathcal X$ associated
to $K$ and
\[\mathcal X_n \xrightarrow[n \to \infty]{
\mathrm{law}} 
\mathcal X.\]
\begin{proof}
See \cite[Proposition 3.10]{ShiraiTakahashi}.
\end{proof}
\end{proposition}

\section{Acknowledgments}
We thank Djalil Chafa\"i and Mathieu Chambefort for their precious help with the simulations. We also thank Raphael Ducatez and Avelio Sep\'ulveda for their help and comments.
Last but not least, we warmly thank the anonymous referee for the many useful suggestions and comments which allowed us to improve this article.

\bibliographystyle{alpha} 
\bibliography{biblio} 
\end{document}